\documentclass{amsart}
\pdfoutput=1

\newcommand{\Z}{\mathbb{Z}}
\newcommand{\R}{\mathbb{R}}
\newcommand{\C}{\mathbb{C}}
\newcommand{\N}{\mathbb{N}}
\newcommand{\vol}{\mathop{\mathrm{vol}}\nolimits}
\newcommand{\supp}{\mathop{\mathrm{supp}}\nolimits}

\usepackage{amssymb}
\usepackage{microtype}
\usepackage{tikz}
\usetikzlibrary{decorations.pathreplacing}
\usepackage[pdftex,colorlinks,citecolor=black,linkcolor=black,urlcolor=black,bookmarks=false]{hyperref}
\def\MR#1{\href{http://www.ams.org/mathscinet-getitem?mr=#1}{MR#1}}
\def\arXiv#1{arXiv:\href{http://arXiv.org/abs/#1}{#1}}
\usepackage{doi}
\usepackage{booktabs}

\newtheorem{theorem}{Theorem}[section]
\newtheorem{proposition}[theorem]{Proposition}
\newtheorem{conjecture}[theorem]{Conjecture}

\newtheorem{corollary}[theorem]{Corollary}
\newtheorem{lemma}[theorem]{Lemma}
\theoremstyle{remark}
\newtheorem{remark}[theorem]{Remark}

\theoremstyle{definition}
\newtheorem{definition}[theorem]{Definition}

\numberwithin{equation}{section}
\numberwithin{table}{section}
\numberwithin{figure}{section}

\title{Three-point bounds for sphere packing}

\author{Henry Cohn}
\address{Microsoft Research New England\\
One Memorial Drive\\
Cambridge, MA 02142\\ USA} \email{cohn@microsoft.com}

\author{David de Laat}
\address{Delft Institute of Applied Mathematics\\
Delft University of Technology\\ Delft, The Netherlands} \email{d.delaat@tudelft.nl}

\author{Andrew Salmon}
\address{Department of Mathematics\\
Massachusetts Institute of Technology\\
Cambridge, MA 02139\\ USA} \email{asalmon@mit.edu}

\date{June 30, 2022}

\thanks{Salmon was supported by an internship at Microsoft Research
New England, and de Laat was supported by Rubicon grant 680-50-1528 of the Netherlands Organisation for Scientific Research.}


\begin{document}

\begin{abstract}
We define three-point bounds for sphere packing that refine the linear programming bound, and we compute these bounds numerically using semidefinite programming by choosing a truncation radius for the three-point function. As a result, we obtain new upper bounds on the sphere packing density in dimensions~$4$ through~$7$ and~$9$ through~$16$. 
We also give a different three-point bound for lattice packing and conjecture that this second bound is sharp in dimension~$4$.
\end{abstract}

\maketitle

\setcounter{tocdepth}{1}
\tableofcontents

\section{Introduction}

Beginning with the work of Delsarte \cite{Delsarte1972}, linear programming bounds have been used to obtain upper bounds for packing problems in discrete geometry.  The analogue of the Delsarte bound for the sphere packing density in Euclidean space was given by Cohn and Elkies \cite{cohn2003new} and yields the best upper bounds on sphere packing known in most dimensions greater than $3$ (see \cite{cohn2014sphere, SardariZargar}).  These bounds are sharp in certain special cases \cite{viazovska2017,CKMRV2017} and give a systematic method of producing upper bounds numerically for a wide range of packing problems. However, they generally fall short of matching the known constructions.

Linear programming bounds are given this name because optimizing such a bound can be viewed as solving a possibly infinite-dimensional linear program. They are also known as two-point bounds, because they can be derived from constraints on the pair correlation functions of packings.

Given the effectiveness but incompleteness of two-point bounds, it is natural to ask for refinements of these bounds.  For example, two-point bounds are not sharp for the kissing problem in three and four dimensions, but Musin \cite{musin2006kissing,musin2008kissing} found a refinement capable of solving these problems, and Pfender \cite{pfender2007improved} generalized Musin's refinement.   The technique of Musin and Pfender produces step functions that are not positive definite kernels and thus are not feasible solutions for the Delsarte linear program, but that are nevertheless copositive and are therefore feasible solutions for the conic dual of the packing problem, formulated as a conic program over the cone of completely positive measures. See \cite{DDFV} for a general framework for such bounds on compact spaces.

At around the same time, Schrijver \cite{Schrijver2005} formulated a semidefinite programming bound for binary error-correcting codes, based not on pair correlations but rather three-point correlations. Bachoc and Vallentin \cite{bachoc2008new} extended this idea to the continuous setting and placed it in a more general context of harmonic analysis, while Gijswijt, Mittelmann, and Schrijver \cite{GMS2012} analyzed four-point interactions for binary codes. The culmination so far of this line of work is de Laat and Vallentin's general framework for higher-order bounds \cite{de2015semidefinite}. In practice, three-point bounds have proved to be numerically feasible and fruitful to analyze, while four-point bounds are far more challenging and have been used in only a few cases \cite{GMS2012,de2018k,deLaat2020,KaoYu}.

Bounds analogous to those of Musin and Pfender have been used by de Laat, Oliveira, and Vallentin \cite{de2014upper} to produce systematic improvements over the linear programming bound for sphere packing.  Their technique uses bounds on spherical codes to produce copositive step functions on Euclidean space, which can be added to the positive definite functions used in the Cohn-Elkies linear program. The resulting bounds improve on the two-point bound in dimensions~$4$ through~$7$ and~$9$ by amounts decreasing from $0.51\%$ to $0.054\%$, and they are the best bounds known in these dimensions.\footnote{Dimension~$8$ is skipped because the two-point bound is already sharp \cite{viazovska2017}.} The same approach \cite{ACHLT} also yields small improvements in all dimensions from $10$ to $179$ except $24$, but the resulting bounds have not been computed explicitly, and it is unclear whether this technique will help in higher dimensions.

In this paper, we develop three-point bounds for sphere packing in Euclidean space.
These bounds improve on those of de Laat, Oliveira, and Vallentin in every dimension in which both bounds have been computed.  In particular, the three-point bounds are the best upper bounds known for sphere packing in dimensions~$4$ through~$7$ and~$9$ through~$16$, 
and we expect that they are strictly stronger than the previously known bounds in every dimension above~$3$ except for the sharp cases~$8$ and~$24$.

No three-point or higher bounds had previously been applied to sphere packing in Euclidean space. As we will see, doing so involves technical obstacles that do not arise in compact spaces.
Truncation plays an essential role in our numerics, via truncated three-point functions that are more flexible than step functions while satisfying the necessary three-point constraints.

We now give a formulation of the truncated three-point bound for sphere packing.  A function $f_3 \colon \R^n \times \R^n \to \R$ is \emph{positive definite as a kernel}, denoted $f_3 \succeq 0$,  if  $f_3(x,y) = f_3(y,x)$ for all $x,y \in \R^n$ and
\begin{equation} \label{eq:pdkineq}
\sum_{x,y \in C} c_x c_y f_3(x,y) \ge 0
\end{equation}
for all finite subsets $C \subseteq \R^n$ and weights $c \in \R^C$. Recall that if $f_3$ is given by $f_3(x,y)= f_2(x-y)$ for some continuous function $f_2 \colon \R^n \to \R$ (in other words, if $f_3$ is continuous and translation invariant), then $f_3$ is positive definite as a kernel if and only if $f_2$ is the Fourier transform of a finite Borel measure on $\R^n$, by Bochner's theorem. Furthermore, if $f_2$ is continuous and integrable, then $f_3 \succeq 0$ iff $\widehat{f_2} \ge 0$. (Note that $\widehat{f_2}$ is automatically integrable under these conditions, by Corollary~1.26 in Chapter~I
of \cite{SteinWeiss}.)
However, being positive definite as a kernel is considerably more general than this translation-invariant special case. For reference, note that we normalize the Fourier transform by
\[
\widehat{f}(y) = \int_{\R^n} f(x) e^{-2\pi i \langle x,y \rangle} \, dx.
\]

\begin{theorem}\label{thm:threepoint}
Let $0 < r < R$, and suppose $f_2 \colon \R^n \to \R$ and $f_3 \colon \R^n \times \R^n \to \R$ are functions such that $f_2$ is continuous and integrable, $\widehat f_2(0) = 1$, $\widehat f_2 \geq 0$ everywhere, and $f_3$ is positive definite as a kernel. If
\[
f_2(x) + f_3(x,0) + f_3(0,x) + f_3(x,x)\leq 0 \text{ for } r \leq |x| \le R,
\]
\[
f_2(x) \leq 0 \text{ for } |x| \ge R,
\] and 
\begin{equation} \label{eq:f3}
f_3(x,y) \le 0 \text{ whenever } r \leq|x|,|y| \leq R \text{ and } |x-y| \ge r,
\end{equation}
then the sphere packing density in $\R^n$ is at most
\[
\vol(B_{r/2}^n) (f_2(0) + f_3(0,0)),
\]
where $B_{r/2}^n$ denotes the closed ball of radius $r/2$ centered at the origin in $\R^n$.
\end{theorem}

Much like the two-point bound from \cite{cohn2003new}, this theorem shows how to obtain sphere packing density bounds from auxiliary functions $f_2$ and $f_3$ satisfying certain inequalities, but it does not tell us which auxiliary functions optimize the resulting bound. Aside from the cases in which the two-point bound is already sharp, we do not know how to optimize this three-point bound exactly. Nevertheless, we can obtain strong bounds from specific choices of $f_2$ and $f_3$.

Theorem~\ref{thm:threepoint} reduces to the two-point bound when $f_3=0$, in which case the truncation radius $R$ becomes irrelevant. In the general case, truncation plays a significant role.
In principle, one could study the untruncated three-point bound with $R = \infty$. However, we do not know how to approximate it numerically. The difficulty is figuring out a suitable finite-dimensional space of functions over which we can optimize the bound. For the two-point bound, polynomials times Gaussians work beautifully, but $f_3$ cannot be of this form when $R = \infty$ unless $f_3=0$:

\begin{lemma}
Suppose $f_3 \colon \R^n \times \R^n \to \R$ is given by $f_3(x,y) = p(x,y) e^{-\alpha(|x|^2+|y|^2)}$, where $p$ is a polynomial and $\alpha\ge 0$, and let $r>0$. If $f_3$ is positive definite as a kernel and satisfies $f_3(x,y) \le 0$ whenever $|x|,|y| \ge r$ and $|x-y| \ge r$, then $f_3$ must be identically zero.
\end{lemma}

\begin{proof}
We first observe that if $f_3$ is positive definite, then $p$ is positive definite, since $e^{\alpha(|x|^2+|y|^2)}$ is positive definite and the pointwise product of positive definite kernels is positive definite. It then follows from Mercer's theorem \cite[Theorem~3.11.9]{SimonVol4} that $p$ can be written as a sum
\begin{equation} \label{eq:polymercer}
p(x,y) = \sum_q q(x)q(y)
\end{equation}
over finitely many polynomials $q$. Specifically, these polynomials $q$ are a basis of eigenfunctions for the map $q \mapsto (x \mapsto \int_{B_1^n(0)} p(x ,y) q(y) \, dy$), and in fact the finite-dimensional spectral theorem can be used instead of Mercer's theorem.

Assuming $p$ is not identically zero, we can choose a unit vector $x_0$ such that at least one polynomial $q$ appearing in the sum \eqref{eq:polymercer} does not vanish identically on the line $\{\lambda x_0 : \lambda \in \R\}$. For each such polynomial $q$, let $\lambda_q x_0$ be its largest root along this line. Then whenever $\lambda, \mu > \max_q \lambda_q$, we must have 
\[
p(\lambda x_0, \mu x_0) = \sum_q q(\lambda x_0) q(\mu x_0) > 0,
\]
which contradicts the requirement that $f_3(x,y) \le 0$ whenever $|x|,|y| \ge r$ and $|x-y| \ge r$.
\end{proof}

Although introducing the truncation parameter $R$ is a simple idea, it is not obvious to us that it is the best way to obtain three-point bounds in Euclidean space. We tried a number of other ideas, each of which failed to improve on two-point bounds numerically, and we therefore consider truncation a key technical contribution to our bounds. However, there are many possibilities we have not investigated thoroughly. Perhaps polynomials times Gaussians are simply not the right sort of auxiliary functions to use for three-point bounds.

One might expect $R$ to tend rapidly to infinity as we increase the degree of the polynomials being used, but that does not happen in practice. Instead, the best choice of $R$ we can find for each given degree is surprisingly small. We do not know what happens in the high-degree limit, but the behavior for low degrees suggests that three-point interactions are most relevant over short distances.

\begin{table}
\caption{The lower bound column shows the best sphere packing density known in dimension $n$, rounded down. The $\Delta^+_3$ column shows rigorous sphere packing upper bounds obtained using Theorem~\ref{thm:threepoint} with truncation radius $R$, while the $\Delta_\mathrm{lat}$ column shows rigorous lattice sphere packing upper bounds obtained using Theorem~\ref{thm:latticethreepoint}. The $\Delta_2$ column is the linear programming bound as computed in \cite{ACHLT}, and $\Delta^\mathrm{str}_2$ is the strengthening given in \cite{de2014upper}.}
\label{table:summary}
\begin{center}
\begin{tabular}{ccccccc}
\toprule
$n$ & Lower bound & $\Delta^+_3$ & $R$ & $\Delta_\mathrm{lat}$ & $\Delta^\mathrm{str}_2$ & $\Delta_2$\\
\midrule
$3$  & $0.740480$ & $0.770271$ & $1.886835$ & $0.741506$ & $0.773080$ & $0.779747$ \\
$4$  & $0.616850$ & $0.636108$ & $1.664634$ & $0.616859$ & $0.644422$ & $0.647705$ \\
$5$  & $0.465257$ & $0.512646$ & $1.501370$ & $0.485080$ & $0.523263$ & $0.524981$ \\ 
$6$  & $0.372947$ & $0.410304$ & $1.470906$ & $0.383938$ & $0.416611$ & $0.417674$ \\
$7$  & $0.295297$ & $0.321148$ & $1.402837$ & $0.309975$ & $0.326921$ & $0.327456$ \\
$8$  & $0.253669$ & --         & --         & --         & --         & $0.253670$ \\
$9$  & $0.145774$ & $0.191121$ & $1.383568$ & $0.193475$ & $0.194451$ & $0.194556$ \\
$10$ & $0.099615$ & $0.143411$ & $1.331593$ & --         & --         & $0.147954$ \\
$11$ & $0.066238$ & $0.106726$ & $1.209355$ & --         & --         & $0.111691$ \\
$12$ & $0.049454$ & $0.079712$ & $1.182417$ & --         & --         & $0.083776$ \\
$13$ & $0.032014$ & $0.060165$ & $1.140159$ & --         & --         & $0.062482$ \\
$14$ & $0.021624$ & $0.045062$ & $1.152887$ & --         & --         & $0.046365$ \\
$15$ & $0.016857$ & $0.033757$ & $1.120467$ & --         & --         & $0.034249$ \\
$16$ & $0.014708$ & $0.024995$ & $1.148582$ & --         & --         & $0.025195$ \\
\bottomrule
\end{tabular}
\end{center}
\end{table}

Table~\ref{table:summary} compares Theorem~\ref{thm:threepoint} with other sphere packing bounds.
The $\Delta^+_3$ column shows the best sphere packing bounds we have obtained from Theorem~\ref{thm:threepoint} using semidefinite programming. It is unclear whether the three-point bound is ever sharp in any case not already resolved by the two-point bound. Our numerical data suggests that this does not happen up through sixteen dimensions, 
although additional optimization could lead to further improvements in low dimensions.

In twelve and sixteen dimensions, the upper bound we achieve for the sphere packing density is less than the dual lower bound proved by Cohn and Triantafillou \cite{cohn2019dual} for the linear programming bound. It follows that the linear programming bound cannot be sharp in those dimensions:

\begin{corollary}
For $n \in \{12,16\}$, no sphere packing in $\R^n$ can have density equal to the $n$-dimensional linear programming bound.
\end{corollary}

In particular, $\R^{16}$ cannot share the remarkable behavior of $\R^8$ and $\R^{24}$ from \cite{viazovska2017,CKMRV2017}, even if a better sphere packing is discovered.
This conclusion presumably holds for all $n > 2$ except for $n=8$ or $24$, but previous bounds left open the possibility that the linear programming bound could be sharp in every dimension. Li \cite{Li} has since proved new dual bounds that rule out sharpness for $n=3$ and $4$ as well.

For the special case of lattice sphere packings we give an alternative three-point bound in Theorem~\ref{thm:latticethreepoint}. This lattice bound is similar in spirit to a theorem of Trinker \cite[Theorem~6.1]{Trinker}, which deals with linear error-correcting codes. As shown in the $\Delta_\mathrm{lat}$ column in Table~\ref{table:summary}, the numerical results suggest this lattice bound is sharp in dimension $4$, and we conjecture its sharpness as Conjecture~\ref{conjectureR4}. It is an open question whether techniques like those introduced by Viazovska \cite{viazovska2017} can be used to construct an optimal function analytically.

\begin{theorem}\label{thm:latticethreepoint}
Let $r> 0$ and define
\[
S_{\mathrm{lat},n} = \big\{ (x,y) \in \R^n \times \R^n : |x|,|y|,|x+y|,|x-y| \in \{0\} \cup [r,\infty)\big\}.
\]
Suppose $f \colon \R^{n} \times \R^n \to \R$ is a continuous, integrable function with $\widehat f \geq 0$ everywhere, $\widehat f(0,0)=1$, and 
\[
f(x,y) \le 0 \text{ for } (x,y) \in S_{\mathrm{lat},n} \setminus \{(0,0)\}.
\]
Then
\[
\vol(B_{r/2}^n) \sqrt{f(0,0)}
\]
is an upper bound on the optimal lattice sphere packing density in $\R^n$.
\end{theorem}

In Section~\ref{sec:truncation}, we give two proofs that the truncated three-point bound in Theorem~\ref{thm:threepoint} is an upper bound on sphere packing, one by using Poisson summation and the other by showing that three-point bounds in Euclidean space are a suitable limit of three-point bounds in compact spaces.  We also show a direct relationship to the three-point bounds for spherical codes, along the same lines as the relationship derived by Cohn and Zhao \cite{cohn2014sphere} in the two-point case.

In Section~\ref{sec:duality}, we prove that certain infinite-dimensional convex optimization problems have no duality gap. These problems include the linear programming bound for sphere packing (which proves a conjecture from \cite{cohn2002new}), the lattice three-point bound, and another problem that arises naturally in Section~\ref{sec:lattice}.

In Section~\ref{sec:lattice}, we prove that the lattice three-point bound in Theorem~\ref{thm:latticethreepoint} is an upper bound on the density of lattice packing.  Furthermore, we use a limiting argument based on the multiplicativity of the Lov\'asz theta prime number for the disjunctive product of graphs to show that the lattice three-point bound is always at least as strong as the two-point linear programming bound in theory.  (In practice, improving on it numerically may require computations with infeasibly high-degree polynomials when the dimension is large.)

In Section~\ref{sec:sdp}, we give an algorithm to numerically compute both the truncated sphere packing and the lattice bounds using semidefinite programming.  The techniques are similar to the three-point bounds for spherical codes, with a space of positive definite functions given by positive definiteness of a matrix and using a sum-of-squares formulation to check inequalities over a semialgebraic set.

Finally, we give tables of results in Section~\ref{sec:numerical}.  In particular, we show the improvements of truncated three-point bounds over existing upper bounds, and we explain how these bounds are rigorously verified.  We also examine the convergence of the lattice three-point bound, and conjecture the sharpness of the bound in dimension~$4$.

\section{Truncated three-point bound for the sphere packing problem}\label{sec:truncation}

\subsection{Truncated Poisson summation}

We begin by obtaining the three-point bound from a variant of the Poisson summation argument used in \cite{cohn2003new}.

\begin{proof}[Proof of Theorem~\ref{thm:threepoint}]
Let $f_2$ and $f_3$ satisfy the hypotheses in the theorem statement.
Because periodic packings come arbitrarily close to the optimal sphere packing density, it will suffice to bound the density of a packing obtained by centering spheres at the union of $m$ distinct translates $\Lambda + x_j$ with $1 \le j \le m$ of a lattice $\Lambda$ in $\R^n$. We can assume without loss of generality that the minimum distance between points in this packing is $r$, in which case we can use spheres of radius $r/2$ and the packing density is $m \vol(B_{r/2}^n)/\vol(\R^n/\Lambda)$.

For each $R \in [0,\infty]$, let $\mathcal{D}_R$ be the multiset consisting of the points $x+x_j-x_k$ with $x \in \Lambda$ and $1 \le j,k \le m$ such that $|x+x_j-x_k| \le R$ (with multiplicity given by the number of choices of $x$, $j$, and $k$ that yield this point); we could call $\mathcal{D}_R$ a truncated difference multiset for the packing.
Recall that the proof of the two-point bound in \cite{cohn2003new} uses the identity
\[
\sum_{u \in \mathcal{D}_\infty} f_2(u) = \frac{1}{\vol(\R^n/\Lambda)} \sum_{w \in \Lambda^*}  \left|\sum_{j=1}^m e^{2\pi i \langle w, x_j \rangle}\right|^2\widehat{f_2}(w),
\]
which follows from Poisson summation. (Strictly speaking, our hypotheses on $f_2$ are not strong enough to justify Poisson summation; we will assume for now that it holds, and deal with this technicality at the end of the proof.) The proof of the two-point bound amounts to observing that due to the inequalities on $f_2$ when $f_3=0$, the left side of the identity is bounded above by the $u=0$ term $m f_2(0)$, while the right side is bounded below by the $w=0$ term $m^2/\vol(\R^n/\Lambda)$. Thus, we obtain the upper bound $\vol(B_{r/2}^n) f_2(0)$ for the density $m \vol(B_{r/2}^n)/\vol(\R^n/\Lambda)$ when $f_3=0$. Furthermore, we always have
\begin{equation} \label{eq:f2ineq}
\sum_{u \in \mathcal{D}_R} f_2(u) \ge \frac{m^2}{\vol(\R^n/\Lambda)},
\end{equation}
regardless of whether $f_3=0$.

To adapt this proof to take $f_3$ into account, we will use the inequality
\begin{equation} \label{eq:f3pdk}
\sum_{j,k,\ell=1}^m \sum_{\substack{x,y \in \Lambda\\|x+x_j-x_k| \le R\\|y+x_j-x_\ell| \le R}} f_3(x+x_j-x_k,y+x_j-x_\ell) \ge 0,
\end{equation}
which holds because the sum for each fixed $j$ is nonnegative by \eqref{eq:pdkineq}.

Because $f_3(u,v) \le 0$ when $r \le |u|,|v| \le R$ and $|u-v| \ge r$, we can remove most of the terms from the sum in
\eqref{eq:f3pdk} and still have a nonnegative result. In particular, for $u = x+x_j-x_k$ and $v = y+x_j-x_\ell$, we have $|u-v| \ge r$ unless $u=v$. (Note that including the same term $x_j$ in both arguments of $f_3$ is needed for this conclusion.) Furthermore, $|u| \ge r$ unless $u=0$, and similarly $|v| \ge r$ unless $v=0$. Removing the terms for which we know $f_3(u,v) \le 0$ from \eqref{eq:f3pdk} therefore yields
\begin{equation} \label{eq:f3pdk2}
m f_3(0,0) + \sum_{u \in \mathcal{D}_R} (f_3(u,0) + f_3(0,u) + f_3(u,u)) \ge 0,
\end{equation}
where the coefficient of $m$ accounts for the $m$ ways to obtain $u=v=0$ (namely, $x=y=0$ and $m$ choices for the value of $j=k=\ell$).

Combining \eqref{eq:f3pdk2} with the two-point inequality \eqref{eq:f2ineq} yields
\[
f_2(0) + f_3(0,0) + \frac{1}{m} \sum_{u \in \mathcal{D}_R \setminus \{0\}} (f_2(u) + f_3(u,0) + f_3(0,u)+ f_3(u,u)) \ge \frac{m}{\vol(\R^n/\Lambda)}.
\]
All the terms on the left except $f_2(0)$ and $f_3(0,0)$ are nonpositive by assumption, and thus we obtain the desired upper bound.

The remaining technicality is how to justify \eqref{eq:f2ineq} using only continuity and integrability for $f_2$. If $f_2$ were smooth and rapidly decreasing, then Poisson summation would hold, but our hypotheses on $f_2$ are insufficient to justify it. Instead, we can mollify $f_2$ as follows. Let $\iota_\varepsilon$ be a nonnegative, smooth, even function of integral~$1$ with support in an $\varepsilon$-neighborhood of the origin, and let $f_{2,\varepsilon} = (f_2 * \iota_\varepsilon * \iota_\varepsilon) \widehat{\iota}_\varepsilon^{\,2}$, where $*$ denotes convolution and thus $\widehat{f}_{2,\varepsilon} = \iota_\varepsilon * \iota_\varepsilon * (\widehat{f}_2 \, \widehat{\iota}_\varepsilon^{\,2})$. Then $f_{2,\varepsilon}$ and $\widehat{f}_{2,\varepsilon}$ are smooth and rapidly decreasing, they converge pointwise to $f_2$ and $\widehat{f}_2$ as $\varepsilon \to 0$, and these functions satisfy $f_{2,\varepsilon}(x) \le 0$ whenever $|x| \ge R + 2\varepsilon$ and $\widehat{f}_{2,\varepsilon} \ge 0$ everywhere. Now Poisson summation for $f_{2,\varepsilon}$ implies that
\[
\sum_{u \in \mathcal{D}_{R + 2\varepsilon}} f_{2,\varepsilon}(u) \ge \frac{m^2}{\vol(\R^n/\Lambda)}\widehat{f}_{2,\varepsilon}(0),
\]
and we obtain \eqref{eq:f2ineq} in the limit as $\varepsilon \to 0$.
\end{proof}

A more abstract perspective on this proof is that each periodic packing yields a feasible point in the conic dual problem. See, for example, \cite{barvinok02} for background on conic programming.

\subsection{Derivation via three-point bounds in compact spaces}\label{3ptscale}

The bound from Theorem~\ref{thm:threepoint} can be strengthened by requiring that
\begin{equation} \label{eq:weaker}
f_3(x,y) + f_3(-x,y-x) + f_3(-y,x-y) \le 0,
\end{equation}
as opposed to 
\[ f_3(x,y) \le 0, \]
for $(x,y)$ such that $|x|,|y|,|x-y| \ge r$ and $|x|,|y| \le R$. The reason \eqref{eq:weaker} strengthens the bound is that the set \[\{(x,y) : \text{$|x|,|y|,|x-y| \in \{0\}\cup [r,\infty)$ and $|x|,|y| \le R$}\},\] is not closed under $(x,y) \mapsto (-x,y-x)$ or $(x,y) \mapsto (-y,x-y)$. Thus, 
\eqref{eq:weaker} involves terms that play no role in Theorem~\ref{thm:threepoint}; these terms could without loss of generality be set equal to zero, while allowing them to be nonzero gives additional flexibility.

In this section we show that the strengthened bound is still a valid upper bound and can be obtained as the limit of the three-point bound for compact spaces. The reason we use Theorem~\ref{thm:threepoint} for our computations is that for the parametrization of $f_2$ and $f_3$ that we use (and the degrees of polynomials for which we can perform computations), the nonstrengthened variant of the bound gives better results. In what follows, we take $r=2$ and $R = \infty$ without loss of generality (one can rescale the functions to achieve $r=2$, and when $R$ is finite the function $f_3$ can be extended by zero), so that we can reuse these variables.

In \cite{de2018k} de Laat, Machado, Oliveira, and Vallentin generalized the three-point bound of Bachoc and Vallentin \cite{bachoc2008new} for spherical codes to a $k$-point bound  for general compact topological packing graphs. Recall that the $k$-point bound is defined as follows for a compact topological packing graph $G$ with vertex set $V$. Let $I_k = I_{k}(G)$ denote the topological space of independent sets of at most $k$ vertices, and let $\mathcal{C}(X)$ be the space of continuous functions from a topological space $X$ to $\R$. We define the operator $B_k \colon \mathcal{C}(V^2 \times I_{k-2}) \to \mathcal{C}(I_{k} \setminus \{ \emptyset \})$ (not to be confused with the notation for a ball) by
\[ B_kT(S) = \sum_{\substack{J \subseteq S \\ |J| \le k-2}} \; \sum_{\substack{x,y \in S \\ \{x,y\} \cup J = S}} T(x,y,J). \]
The \emph{$k$-point bound $\Delta_k(G)$} for the size of the largest independent set in $G$ is the infimum of $M$ over all kernels $K$ in $\mathcal{C}(V^2 \times I_{k-2})$ and real numbers $M$ such that
\begin{enumerate}
\item $(x,y) \mapsto K(x,y,J)$ is positive definite for each fixed $J \in I_{k-2}$,
\item $B_k K(\{ x \}) \le M-1$ and $B_k K(\{ x,y \}) \le -2$ for $\{x, y\} \in I_2$ with $x \ne y$, and
\item for all independent sets $S$ with $|S| \ge 3$, we have $B_k K(S) \le 0$. 
\end{enumerate}
For $k = 2$ this is the Delsarte bound, and for $k=3$ it agrees with the Bachoc-Vallentin three-point bound, except for a $2 \times 2$ matrix in their formulation that does not seem to affect the bound numerically (see Section~5.2 in \cite{de2018k}).

For a compact subset $V$ of $\R^n$, we consider the underlying topological packing graph such that $(x,y)$ is an edge if and only if $0 < |x-y| < 2$. Then independent sets correspond to packings of unit spheres with centers in $V$, and $\Delta_k(V)$ is an upper bound for the maximal number of spheres in such a packing. We will need the following lemma, which is the $\Delta_k$ analogue of a special case of Theorem~3.8 in \cite{rescaling2021}.

\begin{lemma} \label{lemma:subadditive}
Let $V_1$ and $V_2$ be disjoint, compact subsets of $\R^n$.
Then $\Delta_k(V_1 \cup V_2) \le \Delta_k(V_1) + \Delta_k(V_2)$.
\end{lemma}

\begin{proof}
We will adapt the proof of Theorem~3.8 in \cite{rescaling2021}. Suppose $K_1$ and $K_2$ are kernels for $V_1$ and $V_2$ respectively, with corresponding bounds $M_1$ and $M_2$. We can assume that $M_1$ and $M_2$ are positive (i.e., $V_1$ and $V_2$ are nonempty).

We begin by noting that for each $\alpha>0$, the function
\[
L_\alpha := \alpha^2 1_{V_1 \times V_1} + \alpha^{-2} 1_{V_2 \times V_2} - 1_{V_1 \times V_2} - 1_{V_2 \times V_1}
\]
on $\R^n \times \R^n$ is positive definite, because it is the tensor square of $\alpha 1_{V_1} - \alpha^{-1} 1_{V_2}$; here $1_S$ denotes the characteristic function of the subset $S$.

We can view $K_1$ and $K_2$ as elements of $\mathcal{C}((V_1 \cup V_2)^2 \times I_{k-2}(V_1 \cup V_2))$ by extending them by zero. Then we define $K \in \mathcal{C}((V_1 \cup V_2)^2 \times I_{k-2}(V_1 \cup V_2))$ by
\[
K(x,y,J) = (\alpha^2+1) K_1(x,y,J) + (\alpha^{-2}+1) K_2(x,y,J) + \begin{cases}
L_\alpha(x,y) & \text{if $J = \emptyset$, and }\\
0 & \text{otherwise.}
\end{cases}
\]
For each fixed $J$, the function $(x,y) \mapsto K(x,y,J)$ is positive definite, and direct calculation shows that
\begin{align*}
B_kK(\{x\}) &\le (\alpha^2+1)M_1-1 &&\text{if $x \in V_1$},\\
B_kK(\{x\}) &\le (\alpha^{-2}+1)M_2-1 &&\text{if $x \in V_2$},\\
B_kK(\{x,y\}) &\le -2 &&\text{if $x \ne y$, and}\\
B_kK(S) &\le 0 &&\text{if $|S| \ge 3$}.
\end{align*}
Taking $\alpha = \sqrt{M_2/M_1}$ achieves $(\alpha^2+1)M_1=(\alpha^{-2}+1)M_2=M_1+M_2$, as desired.
\end{proof}

We now adapt the definition of $\Delta_k$ to give a Euclidean $k$-point bound.  As with compact spaces, we can consider an underlying graph of $\mathbb{R}^n$ such that $(x,y)$ is an edge if and only if $0 < |x-y| < 2$, and we denote by $I_{k}$ the topological space of independent sets in this graph of size at most $k$.

\begin{definition}\label{def:kpoint}
The \emph{$k$-point sphere packing density bound} $\Delta_{k,\R^n}$ is the infimum of 
\[
\vol(B_{1}^n) B_k f(\{ 0 \})
\]
over all continuous functions $f \colon \R^n \times \R^n \times I_{k-2} \to \R$ satisfying 
\begin{enumerate}
    \item $f(x,y,J) = f(x+z, y+z, J+z)$ for all $x,y,z \in \R^n$ and $J \in I_{k-2}$, where $J+z = \{ j+z : j \in J \}$,
    \item $f(x,y,J)$ is positive definite in $x$ and $y$ for each fixed $J \in I_{k-2}$,
    \item $B_k f(S) \le 0$ for all $|S| \ge 2$, and
    \item $x \mapsto f(x,0,\emptyset)$ is integrable on $\R^n$, and $\int_{\R^n} f(x,0,\emptyset) \, dx = 1$.
\end{enumerate}
\end{definition}

We will first show that $\Delta_{k,\R^n}$ can be interpreted as a limit of $k$-point bounds on compact spaces (which also shows it gives valid sphere packing upper bounds), and then that $\Delta_{3,\R^n}$ is the strengthening of Theorem~\ref{thm:threepoint} discussed in the introduction of this section.

In what follows, for a subset $C \subseteq \mathbb{R}^n$ we let $rC = \{rx : x \in C\}$.

\begin{theorem} \label{thm:limit}
If $C$ is a compact, Jordan-measurable subset of $\mathbb{R}^n$ with nonzero volume, then $\Delta_{k,\R^n} \le \frac{\vol(B_{1}^n)}{\vol(C)}\Delta_k(C)$ and
\[
\Delta_{k, \mathbb{R}^n} = \lim_{r \to \infty} \frac{\vol(B_{1}^n)}{\vol(rC)}\Delta_k(rC).
\]
\end{theorem}

\begin{proof}
This proof will be a generalization of the two-point Delsarte case, which is proved in \cite[Section~4]{rescaling2021}.

Suppose first we are given a $k$-point function $K$ on $C$ that proves an upper bound of $M$ for $\Delta_k(C)$.  We begin by defining an intermediate kernel $\widetilde{f}$ on all of $\mathbb{R}^n$ that is not translation invariant, based on which we will then define a kernel that is translation invariant and is feasible for the Euclidean $k$-point bound.

We define $\widetilde{f} \colon \R^n \times \R^n \times I_{k-2} \to \R$ by
\[
\widetilde{f}(x,y,J) = \begin{cases} \frac{1}{\vol(C)} (K(x,y,\emptyset) + 1) & \text{if $J =\emptyset$ and $x,y \in C$,}\\
\frac{1}{\vol(C)} K(x,y,J) & \text{if $\emptyset\ne J \subseteq C$ and $x,y \in C$, and}\\
0 & \text{otherwise.}\end{cases}
\]
We note the following about this intermediate kernel:
\begin{enumerate}
\item $\widetilde{f}(x,y,J)$ is a positive definite kernel in $x$ and $y$ for any fixed $J \in I_{k-2}$,
\item $B_k \widetilde{f}(\{ x \}) \le M/\!\vol(C)$ for all $x$, 
\item $B_k \widetilde{f}(S) \le 0$ for $|S| \ge 2$, and
\item $\frac{1}{\vol(C)} \iint_{\mathbb{R}^n\times\mathbb{R}^n} \widetilde{f}(x,y, \emptyset) \, dx\, dy\ge 1$.
\end{enumerate}
The first assertion is immediate from the definition and the positive definiteness of $K$, the second and third assertions follow from the corresponding statements for $K$, and the fourth assertion follows from the positive definite property of $K$ applied to the uniform measure on $C$. On the other hand, $\widetilde{f}$ may not be continuous at the boundary of $C$.

Now we define our continuous, integrable, translation-invariant kernel $f$ by 
\[
f(x,y,J) = \frac{1}{\vol(C)} \int_{\mathbb{R}^n} \widetilde{f}(z+x, z+y, z+J)\, dz.
\]
This definition preserves the inequalities $B_k f( \{ x \} ) \le M/\!\vol(C)$ 
and $B_k f(S) \le 0$ for $|S| \ge 2$, and the fourth condition above implies that
\[
\int_{\R^n} f(x,0,\emptyset) \, dx \ge 1.
\]
(If it is greater than $1$, then we can obtain an even better bound by rescaling $f$.)
To check positive definiteness, note that given weights $w_x$ for $x$ in a finite subset $S$ of $\mathbb{R}^n$,
\[ \sum_{x,y \in S} w_x w_y f(x,y,J) = \frac{1}{\vol(C)}\int_{\mathbb{R}^n} \sum_{x,y \in S} w_x w_y \widetilde{f}(z+x, z+y, z+J) \, dz \ge 0. \]

All that remains is the limit as $r \to \infty$, for which it will suffice to show that
\[
\limsup_{r \to \infty} \frac{\vol(B_{1}^n)}{\vol(rC)}\Delta_k(rC) \le \Delta_{k, \mathbb{R}^n}.
\]
Suppose $f$ satisfies the hypotheses of Definition~\ref{def:kpoint}, and we wish to construct a scaling factor $r$ and a kernel $K$ on $rC$ that nearly matches the bound obtained from $f$.
Let $0 < \varepsilon<1$, and choose $R >0$ so that
\begin{equation}\label{eq:secondepsbound}
\int_{\R^n \setminus B_R^n(0)} |f(0,y,\emptyset)| \, dy \le \varepsilon/4.
\end{equation}
Then let $r>R$ be large enough that
\begin{equation}\label{eq:firstepsbound}
\frac{1}{\vol(rC)} \iint_{rC \times rC} f(x,y,\emptyset) \,dx \,dy \ge 1-\varepsilon/2.
\end{equation} 

Let $(rC)_R = \{x \in rC : d(x, \partial rC) \ge R\}$; in other words, $(rC)_R$ consists of the points of $rC$ that are at least a distance $R$ from the boundary. Since $C$ is assumed Jordan measurable, we note that for fixed $R$, \[\lim_{r \to \infty} \frac{\vol((rC)_R)}{\vol(rC)} = 1.\]  
We will construct kernels on $(rC)_R$, which will suffice to prove the limiting result since the ratio of volumes of $(rC)_R$ and $rC$ approaches $1$ as $r \to \infty$.

We define a kernel $K$ on $(rC)_R$ by
\[ K(x,y, J) = \begin{cases}
f(x,y,\emptyset)\vol(rC)/(1+\varepsilon) - 1 & \text{if $J=\emptyset$, and}\\
f(x,y,J)\vol(rC)/(1+\varepsilon) & \text{otherwise}.
\end{cases}
\]
This kernel is continuous, and $(x,y) \mapsto K(x,y,J)$ is by assumption positive definite when $J \ne \emptyset$. Further, the inequalities $B_kK(S) \le 0$ for $|S| \ge 3$, $B_k K(\{ x,y \}) \le -2$, and $B_k K(\{ x \}) \le B_kf(\{0\}) \vol(rC)/(1+\varepsilon) -1$ follow immediately from the hypotheses on $f$. Thus, $K$ satisfies the hypotheses for the $k$-point bound on $(rC)_R$ with $M = B_kf(\{0\}) \vol(rC)/(1+\varepsilon)$, as long as $(x,y) \mapsto K(x,y,\emptyset)$ is positive definite.

To check this final condition, let $S \subseteq (rC)_R$ be an arbitrary finite set with associated real weights $w_x$ for $x \in S$. We must show that
\begin{equation} \label{eq:posdefK}
\sum_{x,y \in S} w_x w_y f(x,y,\emptyset) - \frac{\left( \sum_{x \in S} w_x \right)^2}{\vol(rC)}(1+\varepsilon) \ge 0.
\end{equation}
To do so, let $\nu$ be defined as $\sum_x w_x \delta_x - \mu_r \sum_x w_x/\vol(rC)$, where $\mu_r$ is the Lebesgue measure on $rC$ and $\delta_x$ is the measure of mass $1$ supported on $\{x\}$.  Positive definiteness of $f$ shows that
\[\iint f(x,y, \emptyset) \, d\nu(x) \, d\nu(y) \ge 0,\]
and the left side of this inequality equals
\[
\begin{split}
\sum_{x,y \in S} w_x w_y f(x,y,\emptyset) &+ \frac{\left( \sum_x w_x \right)^2}{\vol(rC)^2} \iint_{rC \times rC} f(x,y,\emptyset) \,dx \,dy\\
&\phantom{}- 2 \frac{\sum_x w_x}{\vol(rC)} \sum_{x' \in S} w_{x'} \int_{rC} f(x',y, \emptyset) \,dy.
\end{split}
\]
Since $S \subseteq (rC)_R$, we can combine this with \eqref{eq:firstepsbound} and \eqref{eq:secondepsbound} to get \eqref{eq:posdefK}, as desired.

Thus, $K$ is a feasible solution for the $k$-point bound on $(rC)_R$ with $M = B_kf(\{0\}) \vol(rC)/(1+\varepsilon)$, and therefore
\begin{align*}
\Delta_k((rC)_R) &\le  B_kf(\{0\}) \frac{\vol(rC)}{1+\varepsilon}\\
&= \frac{\vol(B_{1}^n)B_k f(\{0\})}{\vol(B_{1}^n)} \vol((rC)_R) \frac{\vol(rC)}{(1+\varepsilon) \vol((rC)_R)}
\end{align*}
We can choose $f$ so that $\vol(B_{1}^n)B_k f(\{0\})$ is arbitrarily close to $\Delta_{k, \mathbb{R}^n}$. If $\varepsilon$ is chosen to be small enough and $r$ large enough, then the factor
\[
\frac{\vol(rC)}{(1+\varepsilon) \vol((rC)_R)}
\]
can be made arbitrarily close to $1$. We can therefore conclude that
\[
\limsup_{r \to \infty} \frac{\vol(B_{1}^n)}{\vol((rC)_R)}\Delta_k((rC)_R) \le \Delta_{k, \mathbb{R}^n}.
\]
That inequality is almost what we want, with the only difference being the removal of the points within distance $R$ of the boundary of $rC$. If $C$ is a unit ball centered at the origin, then $(rC)_R = (r-R)C$ and we are done, but more general sets $C$ require additional argument.

Suppose $C$ is a convex set, which we can assume contains the origin. Then $r(1-\varepsilon)C \subseteq (rC)_R$ for $r$ sufficiently large with $\varepsilon$ and $R$ fixed, and so $K$ also gives a bound on packings in $r(1-\varepsilon)C$. It follows that
\begin{align*}
\limsup_{r \to \infty} \frac{\vol(B_{1}^n)}{\vol(rC)}\Delta_k(rC) &= \limsup_{r \to \infty} \frac{\vol(B_{1}^n)}{\vol(r(1-\varepsilon)C)}\Delta_k(r(1-\varepsilon)C)\\
&= \limsup_{r \to \infty} \frac{\vol(B_{1}^n)}{\vol(rC)}(1-\varepsilon)^{-n} \Delta_k(r(1-\varepsilon)C)
\\
&\le (1-\varepsilon)^{-n}\Delta_{k, \mathbb{R}^n},
\end{align*}
and taking $\varepsilon \to 0$ proves the desired inequality when $C$ is convex.

Finally, we can prove the result for Jordan-measurable $C$ by approximating it with cubes. Given $\varepsilon>0$, let $C_1,\dots,C_N$ be disjoint, closed cubes such that $C \subseteq \bigcup_i C_i$ and $\sum_i \vol(C_i) \le \vol(C)(1+\varepsilon)$. Then Lemma~\ref{lemma:subadditive} implies that
\begin{align*}
\limsup_{r \to \infty} \frac{\vol(B_{1}^n)}{\vol(rC)}\Delta_k(rC) &\le 
\sum_{i=1}^N \limsup_{r \to \infty} \frac{\vol(B_{1}^n)}{\vol(rC)}\Delta_k(rC_i)\\
&=  \Delta_{k, \mathbb{R}^n} \sum_{i=1}^N \frac{\vol(C_i)}{\vol(C)}\\
&\le (1+\varepsilon) \Delta_{k, \mathbb{R}^n},
\end{align*}
and again we let $\varepsilon \to 0$.
\end{proof}

\begin{theorem}
The infimum in Theorem~\ref{thm:threepoint} with $f_3(x,y) \le 0$ replaced by \eqref{eq:weaker} is $\Delta_{3,\R^n}$.
\end{theorem}

\begin{proof}
Let $f$ be feasible for $\Delta_{3,\R^n}$. Without loss of generality, we can assume that $f$ is even, i.e., $f(-x,-y,-J)=f(x,y,J)$, since we can replace $f$ with $(x,y,J) \mapsto (f(x,y,J)+f(-x,-y,-J))/2$. Let $f_2(x) = f(x,0,\emptyset)$ and $f_3(x,y) = f(x,y,\{0\})$. Then the inequality $B_3f(S) \le 0$ for the set $S = \{0,x\}$ is equivalent to 
\[
f_2(x) + f_3(x,0) + f_3(0,x) + f_3(x,x)\leq 0 \text{ for $|x| \geq 2$}
\] 
and for $S = \{0,x,y\}$ it is equivalent to \eqref{eq:weaker}. By translation invariance we can always assume $0 \in S$, and so the inequalities in Definition~\ref{def:kpoint} are equivalent to those in Theorem~\ref{thm:threepoint} with  $f_3(x,y) \le 0$ replaced by \eqref{eq:weaker} and with $r=2$ and $R = \infty$, which we can assume without loss of generality.

The only discrepancy is that Theorem~\ref{thm:threepoint} allows $f_3$ to be discontinuous, while Definition~\ref{def:kpoint} does not.
Using the machinery from \cite[Lemma~5.6]{rescaling2021}, one can show that the continuity assumption in Theorem~\ref{thm:threepoint} does not affect the bound. The bound using a continuous $f_2$ and possibly discontinuous $f_3$ sits in between the topological and discrete versions of the three-point bound.  One can show that the discrete $k$-point bound is a sandwich function and can use the same strategy as in \cite[Theorem~6.13]{rescaling2021} to show the necessary inequalities to satisfy the hypotheses of \cite[Lemma~5.6]{rescaling2021}.
\end{proof}

\subsection{Reduction to spherical code bounds}

As evidence that our notion of the three-point bound for sphere packing is the right one, we can relate $k$-point bounds on the sphere with $k$-point bounds on Euclidean space through a generalization of \cite[Theorem~3.4]{cohn2014sphere}.

\begin{theorem}
Let $\pi/3 \le \theta \le \pi$, and let $\Delta_k(S^{n-1}, \theta)$ be the $k$-point bound for spherical codes in $S^{n-1} \subseteq \R^n$ such that any two points must be separated by angle at least $\theta$.  Then
\[ \Delta_{k, \mathbb{R}^n} \le \sin^n(\theta/2) \Delta_k(S^{n-1}, \theta). \]
\end{theorem}

\begin{proof}
We begin with a transformation from independent sets in $\R^n$ (i.e., sets with minimal distance at least~$2$) to those in $S^{n-1}$ (sets with minimal angle at least~$\theta$) via rescaling.
Let $R = 1/\sin(\theta/2)$. By Lemma~2.2 in \cite{cohn2014sphere}, if $J$ is an independent subset of $\R^n$ and $J \subseteq B_R^n(0) \setminus \{0\}$, then the set
\[
\widetilde{J} = \bigg\{ \frac{x}{|x|} : x \in J\bigg\}
\]
is an independent subset of $S^{n-1}$.

Given a kernel $K$ that proves an upper bound of $M$ for the $k$-point bound on $S^{n-1}$ with minimal angle $\theta$, let
\[
F(x,y,J) = K\bigg(\frac{x}{|x|},\frac{y}{|y|}, \widetilde{J}\bigg) + \begin{cases} 1 & \text{if $J = \emptyset$, and }\\
0 & \text{otherwise}
\end{cases}
\]
whenever $J$ is an independent subset of $\R^n$ and $\{x,y\} \cup J \subseteq B_r^n(0) \setminus\{0\}$.
Then we can define a function $f \colon \R^n \times \R^n \times I_{k-2} \to \R^n$ for the $k$-point sphere packing bound by
\[
f(x,y,J) = \int_{A_{\{x,y\} \cup J}} F(x-z,y-z,J-z) \, dz,
\]
where $A_S = \{z \in \R^n : S\subseteq B_R^n(z) \setminus\{z\}\}$.

By construction, $f$ is translation-invariant, and the positive definiteness of $(x,y) \mapsto f(x,y,J)$ for each $J$ follows from the analogous property for $K$. Furthermore, for each nonempty independent subset $S$ of $\R^n$ of size at most $k$,
\[
B_kf(S) = \int_{A_S} \big(B_kK\big(\widetilde{S-z}\big) + c_S\big) \, dz,
\]
where
\[
c_S = \begin{cases} 1 & \text{if $c=1$,}\\
2 & \text{if $c=2$, and}\\
0 & \text{otherwise,}
\end{cases}
\]
and therefore $B_kf(S) \le 0$ for $|S| \ge 2$, while
\[
B_kf(\{0\}) = M \vol(B_R^n).
\]

Finally, we can write
\[
\int_{\R^n} f(x,0,\emptyset) \, dx = \int_{\R^n} \int_{B_R^n(x) \cap B_R^n(0) \setminus \{0,x\}} 1 + K\bigg(\frac{x-z}{|x-z|},\frac{-z}{|z|},\emptyset\bigg) \, dz\, dx.
\]
This integral is finite because the integrand is bounded, as is the domain of integration. The contribution from $K$ is nonnegative, and therefore
\begin{align*}
\int_{\R^n} f(x,0,\emptyset) \, dx &\ge \int_{\R^n} \int_{B_R^n(x) \cap B_R^n(0) \setminus \{0,x\}} dz\, dx\\
&= \vol(\{(x,z) \in \R^n \times \R^n : |z| \le R \text{ and } |x-z| \le R\})\\
&= \vol(B_R^n)^2.
\end{align*}
Rescaling by a factor of $\vol(B_R^n)^2$ therefore yields a function satisfying the conditions in Definition~\ref{def:kpoint} and proving that
\[
\Delta_{k,\R^n} \le \frac{\vol(B_1^n) \vol(B_R^n) M}{\vol(B_R^n)^2} = \sin^n(\theta/2) M,
\]
as desired.
\end{proof}

\section{Duality theory for linear programming bounds}\label{sec:duality}

In our study of the three-point lattice bound, we will use dual bounds for infinite-dimensional linear programs. This theory is subtle, because unlike the finite-dimensional case, there can be a duality gap between the optimal primal and dual bounds. In this section, we will show that there is no duality gap in the linear programs that arise in our work. One consequence of this analysis is that the ordinary linear programming bound has no duality gap, which was conjectured in \cite[Section~4]{cohn2002new} but has not previously been proved.

Let $C$ be a closed subset of $\R^n$. The \emph{$C$-constrained linear programming bound} is the infimum of $f(0)$ over all Schwartz functions $f \colon \R^n \to \R$ satisfying $f(x) \le 0$ for all $x \in C$, $\widehat{f}(y) \ge 0$ for all $y$, and $\widehat{f}(0)=1$. Aside from the restriction to Schwartz functions, which we will address in Proposition~\ref{prop:mollify}, the ordinary linear programming bound for sphere packing is the case $C = \{x \in \R^n : |x| \ge 1\}$, and the three-point lattice bound and the disjunctive product bound from Section~\ref{sec:lattice} are also special cases. Furthermore, when $\R^n \setminus C$ is a convex body that is symmetric with respect to the origin, $f(0) \vol(\R^n \setminus C)/2^n$ is an upper bound for the packing density in $\R^n$ using translates of $\R^n \setminus C$ (see Theorem~B.1 in \cite{cohn2003new}). We will assume $0 \not\in C$, since otherwise no such $f$ exists. Conversely, a feasible $f$ always exists when $0 \not\in C$, for example by rescaling the input variable to a feasible $f$ in the linear programming bound for sphere packing.

The \emph{dual $C$-constrained linear programming bound} is the supremum over all real numbers $c$ such that there exists a tempered distribution $T$ on $\R^n$ that satisfies $T = \delta_0 + \mu$ with $\mu \ge 0$, $\supp(\mu) \subseteq C$, and $\widehat{T} \ge c \delta_0$. For all auxiliary functions $f$ as above and distributions $T$,
\[
f(0) \ge \int T f = \int \widehat{T} \widehat{f} \ge c,
\]
and we say there is no duality gap if the infimum of $f(0)$ and the supremum of $c$ are equal. Note that we make no assertion as to whether the infimum and supremum are attained.

\begin{theorem} \label{theorem:nodualitygap}
For every closed subset $C$ of $\R^n$ with $0 \not\in C$, the $C$-constrained linear programming bound has no duality gap.
\end{theorem}

To prove this theorem, we will adapt a finite-dimensional proof based on hyperplane separation. The main technical obstacle will be showing that the point used for hyperplane separation lies outside of the closure of the convex set encoding the $C$-constrained linear programming bound.

Let $S(\R^n)$ be the Schwartz space on $\R^n$, with the usual topology defined by the Schwartz seminorms, and let $K$ be the closure $\overline{K_0}$ of the convex subset $K_0$ of $S(\R^n) \times S(\R^n)$ given by
\[
\begin{split}
K_0 &= \{(a,b) : \text{there exists $f \in S(\R^n)$ with $\widehat{f}(0)=1$ such that}\\
& \qquad \text{$a(x) \ge f(x)$ for all $x \in C \cup \{0\}$ and $\widehat{f}(y) \ge b(y)$ for all $y \in \R^n$}\}.
\end{split}
\]
This convex set encodes the $C$-constrained linear programming bound in terms of the points $(g,0)$, where $g \in S(\R^n)$ satisfies $g|_C = 0$. If such a point $(g,0)$ lies in $K_0$, then the function $f$ from the definition of $K_0$ is an auxiliary function satisfying $f(0) \le g(0)$.

To prove Theorem~\ref{theorem:nodualitygap}, let $c$ be any real number that is strictly less than the infimum in the $C$-constrained linear programming bound. Then our goal is to construct a tempered distribution $T$ satisfying $T = \delta_0 + \mu$ with $\mu \ge 0$, $\supp(\mu) \subseteq C$, and $\widehat{T} \ge c' \delta_0$ with $c' \ge c$.

Let $g$ be any Schwartz function on $\R^n$ such that $g(0)=c$ and $g|_C = 0$. We have seen that then $(g,0) \not\in K_0$. In fact, we will show that it is not even contained in the closure $K$ of $K_0$. This assertion is the most technical part of the proof. Before checking it, we will first show how to deduce Theorem~\ref{theorem:nodualitygap} from it.

\begin{proof}[Proof of Theorem~\ref{theorem:nodualitygap} assuming $(g,0) \not\in K$]
By hyperplane separation (for example, Theorem~III.3.4 in \cite{barvinok02}), the point $(g,0)$ is separated from $K$ by a closed hyperplane. In other words, there exist tempered distributions $T_1$ and $T_2$ such that for all $(a,b) \in K$,
\[
\int T_1 a + \int T_2 b > \int T_1 g.
\]
Because this inequality holds regardless of what $a$ does in $\R^n \setminus (C \cup \{0\})$, we must have $\supp(T_1) \subseteq \{0\} \cup C$. Furthermore, increasing $a$ or decreasing $b$ does not affect the inequality, and so we must have $T_1 \ge 0$ and $T_2 \le 0$. In particular, $T_1 = \lambda \delta_0 + \mu$ and $T_2 = -\nu$ for some real number $\lambda\ge0$ and nonnegative Radon measures $\mu$ and $\nu$ on $\R^n$ with $\supp(\mu) \subseteq C$. In other words,
\[
\lambda a(0) + \int a \, d\mu - \int b \, d\nu > \lambda c
\]
for all $(a,b) \in K$.

First, we will check that $\lambda$ cannot be zero. If it were, then we would have
\[
\int f \, d\mu > \int \widehat{f} \, d\nu
\]
whenever $f \in S(\R^n)$ satisfies $\widehat{f}(0)=1$, because $(f,\widehat{f}) \in K$. However, that is impossible, since it would rule out the existence of feasible auxiliary functions for the $C$-constrained linear programming bound. Thus, $\lambda > 0$, and by rescaling we can assume that $\lambda = 1$.

Now let $T = \delta_0 + \mu$. Then for $f \in S(\R^n)$,
\begin{equation} \label{eq:Thatineq}
\int (\widehat{T}-\nu) \widehat{f} > c
\end{equation}
whenever $\widehat{f}(0)=1$. Because \eqref{eq:Thatineq} holds regardless of what $\widehat{f}$ does away from the origin, $\widehat{T}-\nu$ must be supported only at the origin, so it is a finite linear combination of derivatives of $\delta_0$. The only way such a linear combination can satisfy \eqref{eq:Thatineq} is if $\widehat{T}-\nu = c' \delta_0$ with $c' > c$, as desired.
\end{proof}

All that remains is to prove that $(g,0) \not\in K$. Suppose instead that a sequence of points $(a_j,b_j) \in K_0$ converges to $(g,0)$ as $j \to \infty$. Each point $(a_j,b_j)$ has a corresponding function $f_j \in S(\R^n)$ satisfying $f_j(x) \le a_j(x)$ for $x \in C \cup \{0\}$, $\widehat{f}_j(y) \ge b_j(y)$ for all $y \in \R^n$, and $\widehat{f}(0)=1$. We can think of these functions $f_j$ as feasible points in a perturbed version of the $C$-constrained linear programming bound, where instead of the inequalities $f_j(x) \le 0$ and $\widehat{f}_j(y) \ge 0$, we perturb the right sides by functions $a_j$ and $b_j$. What we must verify is a form of continuity, namely that these perturbations do not change the optimum in the limit as $j \to \infty$. We will complete the proof with Lemma~\ref{lemma:closure} below.

To analyze the perturbations, we study Schwartz envelopes of rapidly decreasing functions, by using small modifications of the constructions in \cite{Garrett} and \cite[Lemma~18]{Francis} to prove Lemma~\ref{lemma:envelope} below.
We call a function $g \colon \R^n \to \R$ \emph{rapidly decreasing} if $\sup_{x \in \R^n} |x|^k |g(x)| < \infty$ for all integers $k \ge 0$, and we say that a sequence $g_1,g_2,\dots$ of rapidly decreasing functions converges to zero if
\[
\lim_{j \to \infty} \sup_{x \in \R^n} |x|^k |g_j(x)| = 0
\]
for each integer $k \ge 0$. Note that this notion of convergence differs from the Schwartz topology, because it does not take derivatives into account. When we say a sequence of Schwartz functions converges, we will always mean it converges in the Schwartz topology, while this notion will be reserved for functions that are not assumed to be Schwartz functions.

\begin{lemma} \label{lemma:envelope}
For every rapidly decreasing function $g \colon \R^n \to \R$, there exists a Schwartz function $f \colon \R^n \to \R$ such that $f(x) \ge |g(x)|$ for all $x \in \R^n$ and $\supp(\widehat{f}) \subseteq B_1^n(0)$. Furthermore, this construction can be done continuously in the appropriate topologies. In other words, given a sequence $g_1,g_2,\dots$ of rapidly decreasing functions with $g_j \to 0$,
the corresponding functions $f_j$ can be chosen so that $f_j \to 0$ in the Schwartz topology.
\end{lemma}

\begin{proof}
Define $G \colon \R^n \to \R$ by
\[
G(y) = \sup_{|x| \ge |y|-1} |g(x)|.
\]
Then $G$ is measurable because it is radially decreasing, and it is rapidly decreasing because
\[
\begin{split}
\sup_y |y|^k |g(y)| &= \sup_y |y|^k \sup_{|x| \ge |y|-1} |g(x)|\\
&\le \sup_x (1+|x|)^k |g(x)|\\
&\le \sum_{j=0}^k \binom{k}{j} \sup_x |x|^j |g(x)| < \infty.
\end{split}
\]

Now let $h \colon \R^n \to \R$ be a Schwartz function such that $h(x) \ge 0$ for all $x$, $h(0)>0$, and $\supp(\widehat{h}) \subseteq B_1^n(0)$. For example, we can take $h = \widehat{\varphi}^2$, where $\varphi \colon \R^n \to \R$ is a smooth, nonnegative function with $\supp(\varphi) \subseteq B^n_{1/2}(0)$ (and such that $\varphi$ is not identically zero). Given such a function $h$, let $\varepsilon \in (0,1]$ be such that $h(x) \ge \varepsilon$ whenever $|x| \le \varepsilon$, and let $f = (G * h)/(\varepsilon \vol(B^n_\varepsilon(0)))$. Then $f$ is rapidly decreasing because $G$ and $h$ are rapidly decreasing, and the same holds for all derivatives of $f$ because they are convolutions of $G/(\varepsilon \vol(B^n_\varepsilon(0)))$ with the same derivatives of $h$. Thus, $f$ is a Schwartz function. For all $x \in \R^n$,
\[
f(x) = \frac{\int_{\R^n} G(y) h(x-y) \, dy}{\varepsilon \vol(B^n_\varepsilon(0))} \ge
 \frac{\int_{B_\varepsilon^n(x)} G(y) h(x-y) \, dy}{\varepsilon \vol(B^n_\varepsilon(0))} \ge |g(x)|,
\]
because $G(y) \ge |g(x)|$ and $h(x-y) \ge \varepsilon$ whenever $|x-y| \le \varepsilon$. Finally, $\widehat{f} = \widehat{G} \, \widehat{h} / (\varepsilon \vol(B^n_\varepsilon(0)))$, and so $\supp(\widehat{f}) \subseteq \supp(\widehat{h}) \subseteq B^n_1(0)$.

All that remains is to check continuity. Given a sequence $g_1,g_2,\dots$ of functions, we set $G_j(y) = \sup_{|x| \ge |y|-1} |g_j(x)|$, and then $g_j \to 0$ implies $G_j \to 0$. If we take $f_j = (G_j * h)/(\varepsilon \vol(B^n_\varepsilon(0)))$, then the derivatives of $f_j$ all converge to $0$ in the topology on rapidly decreasing functions, and thus $f_j \to 0$ in the Schwartz topology.
\end{proof}

\begin{lemma} \label{lemma:Schwartzconvergence}
For all rapidly decreasing functions $g,h \colon \R^n \to \R$, there exists a Schwartz function $f \colon \R^n \to \R$ such that $f(x) \le -|g(x)|$ for $|x| \ge 1$ and $\widehat{f}(y) \ge |h(y)|$ for all $y$.
Furthermore, this construction can be done continuously: given a sequence of rapidly decreasing functions $g_j,h_j$ with $g_j \to 0$ and $h_j \to 0$, 
the corresponding functions $f_j$ can be chosen so that $f_j \to 0$ in the Schwartz topology.
\end{lemma}

\begin{proof}
By Lemma~\ref{lemma:envelope}, there exist Schwartz functions $g_0,h_0 \colon \R^n \to \R$ such that $g_0(x) \le -|g(x)|$ for all $x$, $\widehat{h}_0(y) \ge |h(y)|+|\widehat{g}_0(y)|$ for all $y$, and $\supp(h_0) \subseteq B^n_1(0)$. Then we can take $f = g_0 + h_0$.
\end{proof}

\begin{lemma} \label{lemma:closure}
Let $C$ be a closed subset of $\R^n$ not containing $0$, let $c$ be strictly less than the infimum in the $C$-constrained linear programming bound, and let $g \in S(\R^n)$ satisfy $g(0)=c$ and $g|_{C} = 0$. Then $(g,0) \not\in K$.
\end{lemma}

\begin{proof}
Without loss of generality, we can assume that $C \subseteq \{x \in \R^n : |x| \ge 1\}$, by rescaling $\R^n$.
Now suppose $(g,0) \in K$. That means there exist sequences $a_j, b_j, f_j$ of Schwartz functions on $\R^n$ such that $a_j(0) \ge f_j(0)$, $a_j(x) \ge f_j(x)$ for $x \in C$, $\widehat{f}_j(y) \ge b_j(y)$ for all $y$, and $\widehat{f}_j(0) = 1$, with $a_j \to g$ and $b_j \to 0$ in the Schwartz topology. By Lemma~\ref{lemma:Schwartzconvergence}, there exist Schwartz functions $g_j \colon \R^n \to \R$ such that $g_j(x) \le -|a_j(x)|$ for $x \in C$, $\widehat{g}_j(y) \ge |b_j(y)|$ for all $y$, and $g_j \to 0$ in the Schwartz topology. Now let $F_j = f_j + g_j$. Then $F_j$ is a Schwartz function such that $F_j(x) \le 0$ for $x \in C$, $\widehat{F}_j(y) \ge 0$ for all $y$, $\widehat{F}_j(0) \to 1$ as $j \to \infty$, and $\liminf_{j \to \infty} F_j(0) = \liminf_{j \to \infty} f_j(0) \le c$, which contradicts the hypothesis that $c$ is strictly less than the infimum in the linear programming bound.
\end{proof}

Now that we have a proof of strong duality, we can deduce that continuous, integrable functions do no better than Schwartz functions in the $C$-constrained linear programming bound:

\begin{proposition} \label{prop:mollify}
Suppose $T$ is a tempered distribution on $\R^n$ such that $T = \delta_0+\mu$ with $\mu \ge 0$ and $\widehat{T} = c \delta_0 + \nu$ with $\nu \ge 0$, and suppose $f \colon \R^n \to \R$ is a continuous, integrable function such that $f(x) \le 0$ for $x \in \supp(\mu)$, $\widehat{f}(y) \ge 0$ for all $y$, and $\widehat{f}(0)=1$. Then $f(0) \ge c$.
\end{proposition}

\begin{proof}
Let $\varphi \colon \R^n \to \R$ be a smooth, even, nonnegative function with $\varphi \ge 0$, $\widehat{\varphi} \ge 0$, $\supp(\varphi) \subseteq B^n_1(0)$, and $\widehat{\varphi}(0)=1$, and for $\varepsilon>0$ let $\varphi_\varepsilon(x) = \varepsilon^{-n} \varphi(x/\varepsilon)$, so that $\widehat{\varphi}_\varepsilon(y) = \widehat{\varphi}(\varepsilon y)$. Then $\varphi_\varepsilon$ forms an approximate identity, i.e., $g * \varphi_\varepsilon \to g$ pointwise as $\varepsilon \to 0$ for every continuous function $g \colon \R^n \to \R$.

Because $f$ is positive definite, $|f|$ is bounded above by $f(0)$. It follows that $f \cdot \widehat{\varphi}_{\varepsilon_1}$ is rapidly decreasing and $h := (f \cdot \widehat{\varphi}_{\varepsilon_1}) * \varphi_{\varepsilon_2}$ is a Schwartz function for every $\varepsilon_1, \varepsilon_2 > 0$.

Every nonnegative tempered distribution is a Borel measure for which $x \mapsto 1/(1+|x|)^k$ is integrable for some $k$, by Chapter~I, \S4, Theorem~V and Chapter~7, \S4, Theorem~VII in \cite{Schwartz}. Thus, $f \cdot \widehat{\varphi}_{\varepsilon_1}$ is integrable with respect to $\mu$.

By the Plancherel theorem,
\[
\int_{\R^n} T h = \int_{\R^n} \widehat{T} \,\widehat{h};
\]
equivalently,
\begin{equation} \label{eq:plancherel}
h(0) + \int h \, d\mu = c\widehat{h}(0) + \int \widehat{h} \, d\nu.
\end{equation}
Because $\widehat{h} = (\widehat{f} * \varphi_{\varepsilon_1}) \cdot \widehat{\varphi}_{\varepsilon_2} \ge 0$, the right side of \eqref{eq:plancherel} is greater than or equal to $c\widehat{h}(0) = c(\widehat{f} * \varphi_{\varepsilon_1})(0)$. The left side of \eqref{eq:plancherel} converges to $f(0) + \int f \cdot \widehat{\varphi}_{\varepsilon_1} \, d\mu$ as $\varepsilon_2 \to 0$ by dominated convergence, since the functions $\varphi_{\varepsilon_2}$ form an approximate identity. (Note that domination holds because $f \cdot \widehat{\varphi}_{\varepsilon_1}$ is rapidly decreasing and $h(x) = ((f \cdot \widehat{\varphi}_{\varepsilon_1}) * \varphi_{\varepsilon_2})(x)$ is a weighted average of $(f \cdot \widehat{\varphi}_{\varepsilon_1})(y)$ over $y \in B^n_{\varepsilon_2}(x)$.) Thus,
\[
f(0) + \int f \cdot \widehat{\varphi}_{\varepsilon_1} \, d\mu \ge c(\widehat{f} * \varphi_{\varepsilon_1})(0).
\]
Now the left side is at most $f(0)$ because $f(x) \le 0$ for $x \in \supp(\mu)$, while the right side converges to $c \widehat{f}(0) = c$ as $\varepsilon_1 \to 0$, again because the functions $\varphi_{\varepsilon_1}$ form an approximate identity. We conclude that $f(0) \ge c$, as desired.
\end{proof}

For the ordinary linear programming bound for sphere packing, one can convert a continuous, integrable auxiliary function to a rapidly decreasing function via direct mollification (see, for example, Section~4 in \cite{cohn2002new}). However, for $C$-constrained linear programming bounds, and in particular the three-point lattice bound, we do not know how to carry out a direct mollification. Fortunately, Proposition~\ref{prop:mollify} provides a substitute.

\begin{proposition} \label{prop:dualattain}
Let $C$ be a closed subset of $\R^n$ with $0 \not\in C$. Then the supremum in the dual $C$-constrained linear programming bound is attained by a tempered distribution.
\end{proposition}

\begin{proof}
Let $T_1,T_2,\dots$ be a sequence of tempered distributions with $\widehat{T}_j \ge c_j \delta_0$, where the sequence $c_1,c_2,\dots$ converges to the dual $C$-constrained linear programming bound $c$. We will obtain a tempered distribution $T$ satisfying the hypotheses of the bound and $\widehat{T} \ge c \delta_0$ via a compactness argument.

Let $U$ be set of Schwartz functions $f \colon \R^n \to \R$ satisfying
\[
\sup_{x \in \R^n} (1+|x|^2)^{(n+1)/2} |f(x)| < 1.
\]
Thanks to the Schwartz seminorms, $U$ is an open subset of $S(\R^n)$. By the Banach-Alaoglu theorem (e.g., Theorem~3.15 in \cite{RudinFA}), the set
\[
U^\circ = \left\{D \in S'(\R^n) : \left|\int D f\right| \le 1 \text{ for all $f \in U$}\right\}
\]
is a compact subset of the space $S'(\R^n)$ of tempered distributions on $\R^n$. Furthermore, it is sequentially compact, since $S(\R^n)$ is separable. (Using a countable, dense subset of $S(\R^n)$, one can put a metric on $S'(\R^n)$ that defines a subset of the usual topology. Then the identity map on $U^\circ$ is a continuous bijection from the usual topology to the metric topology, and it is therefore a homeomorphism. However, compactness is equivalent to sequential compactness in metrizable spaces.)
\end{proof}

We do not know whether there is an analogue of Proposition~\ref{prop:dualattain} for the primal bound. In the special case of the linear programming bound for sphere packing, adapting the proof of Theorem~3 in \cite{gonccalves2017hermite} shows that there exists a continuous, integrable function that attains the optimum (see Proposition~\ref{prop:LPattain}). However, we do not know whether there exists an optimal Schwartz function, or how to generalize this proof to the $C$-constrained linear programming bound.

\section{Three-point bound for the lattice sphere packing problem}\label{sec:lattice}

\begin{proof}[Proof of Theorem~\ref{thm:latticethreepoint}]
By Theorem~\ref{theorem:nodualitygap} and Proposition~\ref{prop:mollify}, Schwartz functions come arbitrarily close to the optimal bound, and so we can assume without loss of generality that $f$ is a Schwartz function.

Let $\Lambda$ be the set of sphere centers of a lattice sphere packing in $\R^n$ by spheres of radius $r/2$. Then $\Lambda$ is a lattice with minimal vector length $r$. By Poisson summation,
\[
\sum_{x,y \in \Lambda} f(x,y) = \frac{1}{\vol(\R^n / \Lambda)^2} \sum_{x,y \in \Lambda^*} \widehat f(x,y) \geq \frac{1}{\vol(\R^n / \Lambda)^2}\widehat f(0) = \frac{1}{\vol(\R^n / \Lambda)^2}.
\]
For $x \in \Lambda \setminus \{0\}$ we have $|x| \ge r$, so $f(x,\pm x) \leq 0$, $f(x,0) \leq 0$, and $f(0,x) \leq 0$. For $x,y \in \Lambda \setminus \{0\}$ with $x \neq \pm y$ we have $|x|,|y|,|x-y|,|x+y| \ge r$, so $f(x,y) \le 0$. It thus follows that
\[
\frac{1}{\vol(\R^n / \Lambda)^2} \le \sum_{x,y \in \Lambda} f(x,y) \le f(0,0),
\]
and the result follows by noting that the sphere packing density of $\Lambda$ is
\[ \frac{\vol(B^n_{r/2})}{\vol(\R^n / \Lambda)}. \qedhere \]
\end{proof}

By symmetrization it follows that the same infimum is obtained if we restrict to functions $f$ satisfying
\[
f(x_1, x_2) = f(s_1 \gamma x_{\sigma(1)}, s_2 \gamma x_{\sigma(2)}).
\]
for all $x_1,x_2 \in \R^{n}$, $\sigma$ in the symmetric group $S_2$, $s \in \{\pm 1\}^2$, and $\gamma$ in the orthogonal group $O(n)$. 

\begin{remark} \label{rem:stronger}
We note that the bound in Theorem~\ref{thm:latticethreepoint} can be strengthened by shrinking the set $S_{\mathrm{lat},n}$ by imposing conditions of the form $|ax + by| \in \{ 0 \} \cup [r, \infty)$ for integers $a$ and $b$.  If we add all such conditions for coprime integers $(a,b)$ to $S_{\mathrm{lat}}$, then the problem acquires a symmetry under the action of $\mathrm{GL}_2(\Z)$.  While one cannot naively average over this infinite group to produce feasible solutions, it is possible to produce dual solutions that are invariant under $\mathrm{GL}_2(\Z)$ by using genus $2$ Siegel modular forms to produce summation formulas.  This is a generalization of how modular forms produce feasible dual solutions to the linear programming problem for sphere packing \cite{cohn2019dual}.
\end{remark}

\subsection{Interpretation via disjunctive products}

In this subsection, we show that the three-point lattice bound strengthens the linear programming bound as an upper bound on lattice packing.  Let $\Delta_2(n)$ be the linear programming bound bound for the $n$-dimensional sphere packing density and $\Delta_{\mathrm{lat}}(n)$ the upper bound on $n$-dimensional lattice packing density given by Theorem~\ref{thm:latticethreepoint}.

\begin{theorem}\label{thm:latticecomparison}
The lattice three-point bound is at least as strong as the two-point bound:
\[ 
\Delta_{\mathrm{lat}}(n) \le \Delta_2(n). 
\]
\end{theorem}

The proof will follow by constructing a bound $\Delta_{\mathrm{prod}}(n, n)$ that is manifestly an upper bound on $\Delta_{\mathrm{lat}}(n)^2$ and is precisely equal to $\Delta_2(n)^2$.  We do this by relating the Delsarte program on compact spaces $G$ and $H$ to the theta prime number $\vartheta'$ of the disjunctive product of these spaces $G * H$, viewed as a topological packing graph.  The Cohn-Elkies program is a scaling limit of the Delsarte program in a precise way, while $\Delta_{\mathrm{prod}}(n, m)$ is a scaling limit of disjunctive product.

Recall that the \emph{disjunctive product} of graphs $G_1$ and $G_2$, denoted $G_1 * G_2$, has vertex set $V(G_1) \times V(G_2)$, and two distinct vertices $(x_1, x_2)$ and $(y_1, y_2)$ are adjacent if $x_1 y_1 \in E(G_1)$ or $x_2 y_2 \in E(G_2)$, inclusively.  In particular, the disjunctive product of topological packing graphs is a topological packing graph.  The name ``disjunctive product'' simply distinguishes it from other graph products by how the conditions $x_1 y_1 \in E(G_1)$ and $x_2 y_2 \in E(G_2)$ are combined.

The independence number $\alpha$ is multiplicative under the disjunctive product:

\begin{proposition}
For all graphs $G_1$ and $G_2$,
\[\alpha(G_1 * G_2) = \alpha(G_1) \alpha(G_2).\]
\end{proposition}

\begin{proof}
We can take the Cartesian product of independent sets to prove that $\alpha(G_1 * G_2) \ge \alpha(G_1) \alpha(G_2)$.  For the other direction, if $C$ is an independent set of the disjunctive product, then the projections $\pi_i(C)$ onto the factors are independent sets, and $C \subseteq \pi_1(C) \times \pi_2(C)$.
\end{proof}

We now prove the analogous result for $\vartheta'$, the theta prime number for compact topological packing graphs \cite{LovaszShannon,MRR,SchrijverComparison,BNOV,de2015semidefinite}. It is noteworthy that this construction lifts to $\vartheta'$ (and, moreover, should lift to the whole Lasserre hierarchy, although we do not consider this).

Recall that for a topological packing graph $G$ with vertex set $V$, $\vartheta'(G)$ considers Borel measures over $V \times V$. The vertex set $V$ embeds diagonally into the product, and its image is denoted $\Delta(V)$.  The theta prime number $\vartheta'(G)$ is defined as the maximum value of $\mu(V \times V)$ over all Borel measures $\mu$ on $V \times V$ that are positive and positive definite, that satisfy $\mu(\Delta(V)) = 1$, and for which the support of $\mu$ consists of pairs $(x,y)$ such that $xy \not \in E(G)$.

The dual bound $\vartheta'(G)^*$ is defined as the infimum of $\max_x K(x,x)$ over all continuous kernels $K$ on $V^2$ such that $K(x,y) - 1$ is a positive definite kernel and $K(x,y) \le 0$ for all $xy \not \in E(G)$.  For infinite graphs, this can be viewed the value of as an infinite-dimensional semidefinite program, whose optimal value is an upper bound on the largest cardinality of an independent set in $V$.  It follows from \cite[Theorem~1]{de2015semidefinite} that the maximum in the definition of $\vartheta'(G)$ is actually achieved and that strong duality holds, i.e., $\vartheta'(G) = \vartheta'(G)^*$.

\begin{theorem} \label{theorem:multiplicative}
The $\vartheta'$ number is multiplicative for the disjunctive product: for compact topological packing graphs $G_1$ and $G_2$,
\[
\vartheta'(G_1 * G_2) = \vartheta'(G_1) \vartheta'(G_2).
\]
\end{theorem}

\begin{proof}
For one direction, taking the outer tensor product of measures shows that $\vartheta'(G_1) \vartheta'(G_2) \le \vartheta'(G_1 * G_2)$.

For the other direction, consider any feasible measure $\mu$ for $G_1 * G_2$.  Let $V_1 = V(G_1)$ and $V_2 = V(G_2)$, and consider the projections $\pi_i \colon V_1 \times V_2 \times V_1 \times V_2 \to V_i \times V_i$ for $i=1,2$. If $\mu$ is feasible for $\vartheta'(G_1 * G_2)$, then the pushforward measure $(\pi_1)_* \mu$ (defined by $((\pi_1)_* \mu)(A) = \mu(\pi_1^{-1}(A))$) is a positive measure, and $(x_1,y_1) \in \mathrm{supp}((\pi_1)_* \mu)$ only when $x_1 y_1 \not \in E$.  Furthermore, it is positive definite, because
\[\int K(x_1, y_1) \,d(\pi_1)_* \mu(x_1, y_1) = \int K(x_1, y_1) 1(x_2,y_2) \,d\mu(x_1,x_2,y_1,y_2) \ge 0\]
for any positive definite kernel $K$, where $1(x_2,y_2)$ denotes the constant function $1$, and so $(\pi_1)_* \mu / \mu(\Delta(V_1) \times V_2 \times V_2)$ is a feasible measure for $\vartheta'(G_1)$.  (Technically the product $\Delta(V_1) \times V_2 \times V_2$ is an abuse of notation, since it reorders the factors relative to $V_1 \times V_2 \times V_1 \times V_2$, but we will identify these spaces as needed to simplify the notation.)

Define $\mu_2$ by $\mu_2(B) = \mu(\Delta(V_1) \times B)$.  For a positive definite kernel $K$ on $V_2 \times V_2$,
\[\int K(x_2, y_2) \,d\mu_2(x_2, y_2) = \int K(x_2, y_2) \Delta(x_1, y_1) \, d\mu(x_1, x_2, y_1, y_2),\]
where 
\[
\Delta(x_1,y_1) = \begin{cases} 1 & \text{if $x_1=y_1$, and}\\
0 & \text{otherwise.}
\end{cases}
\]
This integral is nonnegative, because the Kronecker product of positive definite kernels is positive definite, and therefore $\mu_2$ is feasible as well.

To summarize, we have constructed two measures $(\pi_1)_* \mu / \mu(\Delta(V_1) \times V_2^2)$ and $\mu_2$ that give feasible solutions for $\vartheta'(G_1)$ and $\vartheta'(G_2)$ such that the product of their values is $\mu(V_1^2 \times V_2^2)$.
We conclude that $\vartheta'(G_1) \vartheta'(G_2) \ge \vartheta'(G_1 * G_2)$, as desired.
\end{proof}

The linear programming bound for sphere packing is the scaling limit of $\vartheta'$ in Euclidean space of dimension $n$ by Theorem~5.7 in \cite{rescaling2021}, and it is natural to consider the scaling limit of $\vartheta'$ for disjunctive products $\vartheta'(G_1 * G_2)$.

\begin{definition} \label{def:Deltaprod}
Let $\Delta_{\mathrm{prod}}(m, n)$ be the infimum of
\[ \vol(B^m_{1}) \vol(B^n_{1}) \frac{f(0,0)}{\widehat{f}(0,0)} \]
over all continuous, integrable functions $f \colon \R^{m} \times \R^n \to \R$ such that
\begin{enumerate}
    \item $f$ is positive definite, i.e, $\widehat{f}(x,y) \ge 0$ for all $x \in \R^m$ and $y \in \R^n$,
    \item $\widehat{f}(0,0)>0$,
    \item $f(x,y) \le 0$ for all $x$ and $y$ such that $|x| \ge 2$ and $|y| \ge 2$, and
    \item $f(x,0) \le 0$ and $f(0,y) \le 0$ for all $|x| \ge 2$ and $|y| \ge 2$.
\end{enumerate}
\end{definition}

Note that without loss of generality, we can assume these functions $f$ are invariant under $O(m) \times O(n)$ (by averaging under these rotations), because the constraints and objective functions are invariant under this action. Furthermore, we can assume that $f$ has compact support by mollifying a feasible function $f$ as follows. The convolution $1_{B_{R/2}^n(0)} * 1_{B_{R/2}^n(0)}$ is a continuous function supported in $B_R^n(0)$, and it is positive semidefinite since its Fourier transform is $\big(\widehat{1}_{B_{R/2}^n(0)}\big)^2$. If we normalize it by setting $g_{R,n} := \big(1_{B_{R/2}^n(0)} * 1_{B_{R/2}^n(0)}\big)/\vol(B_{R/2}^n(0))$ so that $g_{R,n}(0)=1$, then $g_{R,n}$ converges pointwise to $1$ everywhere as $R \to \infty$. Now the product $f_R := f \cdot g_{R,m} g_{R,n}$ is continuous and supported in $B_R^m(0) \times B_R^n(0)$, it is positive semidefinite by the Schur product theorem (Theorem~7.5.3 in \cite{HJ}), and $\lim_{R \to \infty} \widehat{f}_R(0) = \widehat{f}(0)$ by dominated convergence. By replacing $f$ with $f_R$, we can come arbitrarily close to the ratio $f(0)/\widehat{f}(0)$ using compactly supported functions.

Any function $f$ that satisfies the hypotheses of Definition~\ref{def:Deltaprod} with $m=n$ automatically satisfies the hypotheses of Theorem~\ref{thm:latticethreepoint} with $r=2$. Thus, we obtain the following upper bound for $\Delta_{\mathrm{lat}}(n)$:

\begin{proposition}
The lattice three-point bound satisfies
\[
\Delta_{\mathrm{lat}}(n)^2  \le \Delta_{\mathrm{prod}}(n, n).
\]
\end{proposition}

Theorem~\ref{thm:latticecomparison} follows immediately from this upper bound and the following proposition:

\begin{proposition} \label{prop:Deltaprod}
For all positive integers $m$ and $n$,
\[
\Delta_{\mathrm{prod}}(m,n) = \Delta_2(m) \Delta_2(n).
\]
\end{proposition}

To prove Proposition~\ref{prop:Deltaprod}, we will need the following lemma, which shows that $\Delta_{\mathrm{prod}}$ is related to $\vartheta'(C_1 * C_2)$ for compact spaces $C_1$ and $C_2$ in the same way that $\Delta_2$ is related to $\vartheta'(C_1)$ or $\vartheta'(C_2)$ individually.
Recall that we give a compact subset $C$ of $\R^n$ a topological packing graph structure by letting $xy \in E(C)$ if and only if $0 < |x-y| < 1$.

\begin{lemma} \label{lemma:thetaprimedisj}
Let $C_1$ and $C_2$ be compact, Jordan-measurable subsets of $\mathbb{R}^m$ and $\mathbb{R}^n$, respectively.  Then
\[\vol(B^m_{1}) \vol(B^n_{1}) \frac{\vartheta'(C_1*C_2)}{\vol(C_1) \vol(C_2)} \ge \Delta_{\mathrm{prod}}(m,n).\]
\end{lemma}

\begin{proof}
Let $K$ be a feasible kernel for $\vartheta'(C_1 * C_2)$ satisfying $K((x_1, x_2), (x_1, x_2)) \le M$ for $x_i \in C_i$, and define $f \colon (\R^m \times \R^n)^2 \to \R$ by $f(x,y) = K(x,y)$ when this is defined and $0$ otherwise.  To produce a feasible function for $\Delta_{\mathrm{prod}}(m,n)$, we define $g \colon \R^m \times \R^n \to \R$ by
\[
g(x) = \int_{\mathbb{R}^{m+n}} f(z+x,z) \,dz.
\]

To show that $g$ is feasible, we first note that it is positive definite because $f$ is; specifically,
\[
\sum_{j,k} c_j c_k g(x_j-x_k) = \int_{\R^{m+n}} \sum_{j,k} c_j c_k f(z+x_j,z+x_k) \, dz \ge 0.
\]
The inequality $g(x_1,x_2) \le 0$ if both $|x_1| \ge 2$ and $|x_2| \ge 2$ follows from the corresponding fact for $f(z+x, z)$, since in that case $(z+x, z)$ is never an edge of the disjunctive product graph $C_1 * C_2$. Similarly, $g(x_1,0) \le 0$ and $g(0,x_2) \le 0$.

The solution satisfies
\[g(0,0) =\int f(z,z) \,dz \le M \vol(C_1) \vol(C_2)\]
and
\[\widehat{g}(0, 0) = \int_{\mathbb{R}^{m+n}} g(x) \,dx = \iint f(x,y) \,dx \,dy \ge \vol(C_1) \vol(C_2),\]
because $K-1$ is by assumption a positive definite kernel on $C_1 \times C_2$.
Thus, $\vol(B^m_{1}) \vol(B^n_{1}) M$ is an upper bound for $\Delta_{\mathrm{prod}}(m,n)$.
\end{proof}

\begin{proof}[Proof of Proposition~\ref{prop:Deltaprod}]
To prove that $\Delta_2(m) \Delta_2(n) \le \Delta_{\mathrm{prod}}(m,n)$, we will use the framework of Section~\ref{sec:duality}. If we let
\[
C = \{(x,y) \in \R^m \times \R^n : |x|,|y| \in \{ 0 \} \cup [1,\infty)\},
\]
then $\Delta_{\mathrm{prod}}(m,n)$ is the $C \setminus \{ (0,0) \}$-constrained linear programming bound.  A feasible dual solution is a measure $\mu$ supported on $C$ such that $\mu(\{(0,0)\}) = 1$ and $\widehat{\mu} \ge c \delta_{(0,0)}$.  In particular, feasible solutions $\nu_1$, $\nu_2$ for the dual linear programming bounds in $\R^m$ and $\R^n$ give a feasible solution for the dual of $\Delta_{\mathrm{prod}}(m,n)$ by letting $\mu = \nu_1 \otimes \nu_2$.   Because there is no duality gap (Theorem~\ref{theorem:nodualitygap}), this shows that $\Delta_2(m) \Delta_2(n) \le \Delta_{\mathrm{prod}}(m,n)$.

To prove the other direction, we will combine Lemma~\ref{lemma:thetaprimedisj} with Theorem~5.7 from \cite{rescaling2021}, which tells us that
\[
\Delta_2(n) = \vol(B_1^n) \lim_{r \to \infty} \frac{\vartheta'(rI^n)}{r^n},
\]
where $I^n$ is a unit cube in $\R^n$ and as usual $rI^n = \{rx : x \in I^n\}$. By Lemma~~\ref{lemma:thetaprimedisj},
\[
\vol(B^m_{1}) \vol(B^n_{1}) \frac{\vartheta'(r I^m*r I^n)}{r^{m+n}} \ge \Delta_{\mathrm{prod}}(m,n).
\]
By Theorem~\ref{theorem:multiplicative},
\[
\vartheta'(r I^m*r I^n) = \vartheta'(r I^m) \vartheta'(r I^n),
\]
and taking the limit as $r \to \infty$ shows that $\Delta_2(m) \Delta_2(n) \ge \Delta_{\mathrm{prod}}(m,n)$.
\end{proof}

\begin{remark}
For $m,n \in \{1,8,24\}$, the exact optimum is known for the linear programming bound, and therefore $\Delta_{\mathrm{prod}}(m,n)$ can also be determined exactly. We do not know whether there is a single feasible function that attains the optimum in each case. If so, it would be interesting to construct these functions explicitly.
\end{remark}

\section{Semidefinite programming formulations}\label{sec:sdp}

To obtain numerical results, we produce good feasible solutions to the three-point and lattice bounds using semidefinite programming. To do so, we parametrize the functions by polynomials times Gaussians and use sums of squares to model the inequality constraints. In what follows, $L_n^{\alpha}$ is the Laguerre polynomial of degree $n$ with parameter $\alpha$ (see, for example, \cite[Section~6.2]{andrews1999special}), and $P_\ell^n$ is the Jacobi polynomial $P_\ell^{((n-3)/2,(n-3)/2)}$ of degree $\ell$ with parameters $((n-3)/2,(n-3)/2)$ (see \cite[Section~6.3]{andrews1999special}).

\subsection{Sum-of-squares formulation for the truncated three-point bound}\label{sec:param-trunc}

We parametrize the two-point function $f_2$ by 
\[
f_2(x) = p_2(|x|^2) e^{-\pi |x|^2}
\]
for a single-variable polynomial $p_2$ with
\[
p_2(w) = -s_1(w) - (w - R^2) s_2(w),
\]
where $s_1$ and $s_2$ are sum-of-squares polynomials of degree $2d_2-2$. By construction $f_2(x)$ is nonpositive for $|x| \geq R$. To ensure that the Fourier transform of $f_2$ is nonnegative we add linear constraints (by equating coefficients in the monomial basis) for the identity
\[
\mathcal F(p_2) = s_3(u) + w \, s_4(u),
\]
where $s_3$ and $s_4$ are sum-of-squares polynomials of degree $2d_2-2$, and where $\mathcal F$ is the linear operator defined by
\[
\mathcal F(u^k) = \frac{k!}{\pi^k} L_k^{n/2-1}(\pi u),
\]
so that $x \mapsto \mathcal F(p)(|x|^2) e^{-\pi |x|^2}$ is the Fourier transform of $x \mapsto p(|x|^2) e^{-\pi|x|^2}$.

For the parametrization of the three-point function $f_3$ we use the following lemma.

\begin{lemma}\label{polynomials}
Let $p \in \R[x,y]$ be a polynomial of degree $4d_3-2$ in variables $x_1,\dots,x_n,y_1,\dots,y_n$. Suppose $p$ is positive definite as a kernel on $\R^n \times \R^n$ and is invariant under the diagonal action of $O(n)$. Then there are positive semidefinite matrices $A^{(0)},\dots,A^{(2d_3-1)}$ such that
\[
p(x,y) = \sum_{\ell=0}^{2d_3-1} \sum_{k,k'=0}^{d_3-\lfloor \ell/2 \rfloor-1} A_{k,k'}^{(\ell)} u^{k+\ell/2} v^{k'+\ell/2} P_\ell^n\mathopen{}\left(\frac{t}{\sqrt{uv}}\right)\mathclose{},
\]
where $u=|x|^2$, $v= |y|^2$, and $t = \langle x, y\rangle$.
\end{lemma}

Here $A^{(\ell)}$ is a $(d_3-\lfloor \ell/2 \rfloor) \times (d_3-\lfloor \ell/2 \rfloor)$ matrix, whose entries are indexed starting with $0$.

\begin{proof}
Consider the operator $T_p \colon \mathcal C([0,1]^n) \to \mathcal C([0,1]^n)$ defined by
\[
T_p f(x) = \int_{[0,1]^n} p(x,y) f(y) \, dy. 
\]
Since $p$ defines a continuous, positive definite kernel $[0,1]^n \times [0,1]^n \to \R$, by Mercer's theorem \cite[Theorem~3.11.9]{SimonVol4} there exists an orthonormal basis $\{e_i\}$ of $L^2([0,1]^n)$ consisting of eigenfunctions of $T_p$ with nonnegative eigenvalues $\lambda_i$. The eigenfunctions corresponding to nonzero eigenvalues are continuous, and 
\begin{equation}\label{eq:mercer}
p(x,y) = \sum_{\lambda_i >0} \lambda_i e_i(x) e_i(y),
\end{equation}
where convergence is uniform and absolute.

Since $p$ is a polynomial of degree $4d_3-2$, the image of $T_p$ consists of polynomials of degree at most $4d_3-2$, and so the eigenfunctions corresponding to nonzero eigenvalues are polynomials of degree at most $4d_3-2$. This implies the terms on the right hand side of \eqref{eq:mercer} are polynomials of degree at most $4d_3-2$, which shows that \eqref{eq:mercer} holds for all $x, y \in \R^n$. Since $p(x,x)$ has degree $4d_3-2$, it follows that each $e_i$ has degree at most $2d_3-1$.  

Let $Y_\ell^1,\dots,Y_\ell^{m_\ell}$ be a basis for the real spherical harmonics of degree $\ell$, with
\[
\int_{S^{n-1}} Y_{\ell}^m(x) Y_{\ell'}^{m'}(x) \, dx = \delta_{\ell,\ell'} \delta_{m,m'}.
\]
Since $e_i$ is a polynomial of degree $2d_3-1$, it is a linear combination of functions of the form
\[
\phi_{\ell,m,k}(x) = |x|^{2k+\ell} Y_{\ell}^m\mathopen{}\left(\frac{x}{|x|}\right)\mathclose{}
\]
for nonnegative integers $k$ and $\ell$ satisfying $2k+\ell \le 2d_3-1$ and $m=1,\dots,m_\ell$. 
The conditions on $k$ and $\ell$ amount to $0\leq \ell \leq 2d_3-1$ and $0 \leq k \leq d_3-\lfloor \ell/2 \rfloor - 1$. This shows $p$ is of the form
\[
p(x,y) = \sum_{\ell,\ell'=0}^{2d_3-1} \sum_{m=1}^{m_\ell} \sum_{m'=1}^{m_{\ell'}} \sum_{k=0}^{d_3-\lfloor \ell/2 \rfloor - 1} \sum_{k'=0}^{d_3-\lfloor \ell'/2 \rfloor - 1}  A_{(\ell,m,k),(\ell',m',k')} \phi_{\ell,m,k}(x) \phi_{\ell',k',m'}(y)
\]
for some positive semidefinite matrix $A$.

By $O(n)$-invariance of $p$,
\[
p(x,y) = \sum_{\ell,m,k}  \sum_{\ell',m',k'} A_{(\ell,m,k),(\ell',m',k')} \frac{\int_{O(n)} \phi_{\ell,m,k}(\gamma x) \phi_{\ell',k',m'}(\gamma y)\, d\gamma}{\int_{O(n)} d\gamma},
\]
and so the result follows from the identity
\[
\int_{O(n)} Y_\ell^m(\gamma x) Y_{\ell'}^{m'}(\gamma y)\, d\gamma = c_\ell  \delta_{\ell,\ell'} \delta_{m,m'} P_\ell^n(\langle x, y\rangle),
\]
which holds for some constant $c_\ell$.
\end{proof}

Define $f_3$ by 
\[
f_3(x,y) = p_3(|x|^2, |y|^2, \langle x, y \rangle) e^{-\pi (|x|^2 + |y|^2)},
\]
where
\[
p_3(u,v,t) = \sum_{\ell=0}^{2d_3-1} \sum_{k,k'=0}^{d_3-\lfloor \ell/2 \rfloor-1} A_{k,k'}^{(\ell)} u^{k+\ell/2} v^{k'+\ell/2} P_\ell^n\mathopen{}\left(\frac{t}{\sqrt{uv}}\right)\mathclose{},
\]
and where $A^{(\ell)}$ is a positive semidefinite matrix for each $\ell$. 

Let $\phi \colon \R^{2n} \to \R^3$ be the map $\phi(x,y) = (|x|^2, |y|^2, \langle x, y\rangle)$. In the following lemma we give polynomials describing the semialgebraic set 
\[
S = \phi(\{(x,y) : r \leq|x|,|y| \leq R,\, |x-y| \ge r\}).
\]
These polynomials are invariant under permutation of $u$ and $v$, which  allows us to use invariant sum-of-squares polynomials (see Section~\ref{sec:sdpformulation}).

\begin{lemma}
The set $S$ consists of the tuples $(u,v,t)$ satisfying 
\begin{align*}
&(u-r^2)(v-r^2) \ge 0,\\
&u+v - 2r^2 \ge 0,\\
&uv - t^2 \ge 0,\\
&u+v-2t-r^2 \ge 0,\\
&(R^2-u)(R^2-v) \ge 0, \text{and}\\
&2R^2 - u - v \ge 0.
\end{align*}
\end{lemma}

\begin{proof}
If $x,y \in \R^n$ with $r \leq|x|,|y| \leq R$ and $|x-y| \ge r$, then it is straightforward to check that
\[
(|x|^2, |y|^2, \langle x, y\rangle) \in S.
\]
Conversely, if $(u,v,t) \in S$, then the first two constraints imply that $u, v \ge r^2$. The third constraint then implies that there exist points $x,y \in \R^n$ with $u = |x|^2$, $v = |y|^2$, and $t = \langle x, y\rangle$. The fourth constraint implies that $|x-y| \geq r$, and the final two constraints imply that $|x|, |y| \leq R$.
\end{proof}

The constraint $f_3(x,y) \leq 0$ for $r \leq |x|,|y| \leq R$ and $|x-y| \ge r$ is equivalent to $p_3(u,v,t) \leq 0$ for $(u,v,t) \in S$, which can be modeled by adding  linear constraints that ensure the identity 
\begin{align*}
p_3(u,v,t) &= - S_1(u,v,y) - (u-r^2)(v-r^2)S_2(u,v,t) - (u+v-2r^2)S_3(u,v,t)\\
&\qquad - (uv-t^2) S_4(u,v,y) - (u+v-2t-r^2)S_5(u,v,t) \\
&\qquad- (R^2-u)(R^2-v)S_6(u,v,t) - (2R^2 -u-v)S_7(u,v,t)
\end{align*}
holds,  where $S_1,\dots,S_7$ are sum-of-squares polynomials with $S_i(u,v,t) = S_i(v,u,t)$ such that $S_1$ has total degree $2d_3$ and $S_2,\dots,S_7$ have total degree $2d_3-2$.

The constraint 
\[
f_2(x) + f_3(x,0) + f_3(0,x) + f_3(x,x) \leq 0
\]
for $r \leq |x| \leq R$ is equivalent to 
\begin{equation} \label{eq:nonpoly}
p_2(u) + p_3(u,0,0) + p_3(0,u,0) + p_3(u,u,u) e^{-\pi u} \leq 0
\end{equation}
for $r^2 \leq u \leq R^2$. This is not a polynomial inequality, but we can choose a nonnegative integer $d_{\mathrm{taylor}}$ and use the inequality
\[
e^{-\pi u} \le T(u) := \sum_{k=0}^{2d_\mathrm{taylor}} \frac{(-\pi)^k}{k!} e^{-\pi r^2} (u-r^2)^k,
\]
which holds for $u\ge r^2$.
Since $p_3(u,u,u) \ge 0$ (the diagonal of a positive definite kernel is nonnegative), the constraint \eqref{eq:nonpoly} is implied by 
\[
p_2(u) + p_3(u,0,0) + p_3(0,u,0) + p_3(u,u,u) T(u) = -s_5(u) - us_6(u),
\]
where $s_5$ and $s_6$ are sum-of-squares polynomials of degree $\max\{2d_2, 2d_3+2d_\mathrm{taylor}\}$ and $\max\{2d_2, 2d_3+2d_\mathrm{taylor}\}-2$, respectively.

\subsubsection{Alternative parametrizations}

For Theorem~\ref{thm:threepoint} the values of $f_3$ outside of the compact region $\{(x,y) : |x|, |y| \leq R\}$ are not relevant. This means we could model $f_3$ by a polynomial $p_3(|x|^2,|y|^2,\langle x,y\rangle)$, as opposed to a polynomial times a Gaussian, while keeping $f_2(x) = p_2(|x|^2) e^{-\pi |x|^2}$. Modeling the constraint
\[
f_2(x) + f_3(x,0) + f_3(0,x) + f_3(x,x) \leq 0,
\]
however, becomes less convenient: in the computations the degree of $p_2$ will be much higher than the degree of $p_3$, and for this reason it is preferable to multiply the constraint by $e^{\pi |x|^2}$, which leads to the constraint
\[
p_2(u) + (p_3(u,0,0) + p_3(0,u,0) + p_3(u,u,u)) e^{\pi u} \leq 0.
\]
We can then use a Taylor expansion of $e^{\pi u}$, and although this will not give rigorous upper bounds, we can use it to obtain a good indication of the strength of this bound. In dimension $3$ with $d_2 = d_\mathrm{taylor} = 24$, $d_3=7$, and $r=1$ we get the bound $0.774026$ after optimizing for $R$. This is worse than the bound $0.772312$ we get with the method described earlier for the same parameters.

For the alternative three-point bound using the weaker constraint 
\[
f_3(x,y) + f_3(-x,y-x) + f_3(-y,x-y) \leq 0
\]
we also parametrize $f_3$ by a polynomial as opposed to a polynomial times a Gaussian, and computations with the same parameters as above give the bound $0.773862$. 

Since these two approaches seem to give worse results, we did not investigate them further.

\subsection{Sum-of-squares formulation for the lattice three-point bound}\label{sec:parlat}

To find good functions satisfying the hypothesis of Theorem~\ref{thm:latticethreepoint}, we optimize over functions of the form
\[
f(x,y) = p(|x|^2, |y|^2, \langle x, y\rangle^2) e^{- \pi(|x|^2 + |y|^2)},
\]
where $p(u,v,t) \in \bigoplus_{k=0}^{d} t^k \R[u,v]_{2d-2k}$ satisfies $p(u,v,t) = p(v,u,t)$ for all $u,v,t$. Here $\R[u,v]_s$ is the set of polynomials in $u$ and $v$ of total degree at most $s$.

\begin{lemma}
Define the map $\phi \colon \R^n \times \R^n \to \R^3$ by $\phi(x,y) = (|x|^2, |y|^2, \langle x, y \rangle^2)$, and consider the inequalities
\begin{enumerate}
\item $(u-r^2)(v-r^2) \ge 0$, 
\item $u+v-2r^2 \ge 0$,
\item $t(uv-t) \ge 0$,
\item $(u+v-r^2)^2-4t\ge 0$,
\item $(u+4v-r^2)^2+(4u+v-r^2)^2-32t\ge 0$,
\item $((u+4v-r^2)^2-16t)((4u+v-r^2)^2-16t)\ge 0$.
\end{enumerate}
For $r > 0$, the set $\phi(\{(x,y) : |x|, |y|, |x+y|, |x-y| \geq r\})$
consists of the points $(u,v,t)$ satisfying inequalities (1)--(4), and the set
\[
\phi(\{(x,y) : |x|, |y|, |x\pm y|, |x\pm 2y|, |2x \pm y| \geq r\})
\]
consists of the points $(u,v,t)$ satisfying inequalities (1)--(6).
\end{lemma}

Inequalities (5) and (6) are not needed for the lattice three-point bound from Theorem~\ref{thm:latticethreepoint}, but they become relevant for refinements of this bound.

\begin{proof}
If $ |x|, |y|, |x\pm y| \geq r$, then it follows immediately that inequalities (1)--(4) are satisfied for $(u,v,t) = \phi(x,y)$, and if in addition $|x \pm 2y|, |2x\pm y| \geq r$, then inequalities (5)--(6) are also satisfied. 

For the other direction, we first observe that the first two inequalities together imply $u \ge r^2$ and $v \ge r^2$. The third inequality then implies there exist points $x,y \in \R^n$ such that $u = |x|^2$, $v = |y|^2$, and $t = \langle x, y \rangle^2$. The fourth inequality gives $4 \langle x, y\rangle^2 \leq (|x|^2 + |y|^2 -r^2)^2$, which implies 
\[
|x|^2 + |y|^2 \pm 2 \langle x, y\rangle^2 \geq r^2
\]
since $|x|^2 + |y|^2 -r^2$ is nonnegative. This implies $|x+y|, |x-y| \geq r$.

Inequalities (5)--(6) together imply $(u+4v-r^2)^2 \geq 16t$ and $(4u+v-r^2)^2 \geq 16t$, which shows that $|x \pm 2y|, |2x \pm y| \geq r$. 
\end{proof}

For $f$ to satisfy the constraint $f(x,y) \leq 0$ for $|x|, |y| \geq r$ and $|x\pm y|\geq r$, we define $p$ by
\begin{align*}
p(u,v,t) &:= - s_1(u,v,t) - (u-r^2)(v-r^2) s_2(u,v,t)- (u+v-2r^2) s_3(u,v,t)\\
&\qquad - t(uv-t)s_4(u,v,t) - ((u+v-r^2)^2-4t) s_5(u,v,t),
\end{align*}
where $s_1,\dots,s_5$ are sum-of-squares polynomials of appropriate degrees that satisfy $s_i(u,v,t) = s_i(v,u,t)$ for all $u,v,t$. 

For $f$ to satisfy $f(x,x), f(x,-x),f(x,0), f(0,x) \le 0$ for $|x| \ge r$ we require that
\[
p(u,u,u^2) = - (u-r^2) s_6(u) 
\]
and
\[
p(u,0,0) = - (u-r^2) s_7(u) 
\]
where $s_6$ and $s_7$ are sum-of-squares polynomials.

For the constraint $\widehat f \geq 0$ we use the following lemma.

\begin{lemma}\label{lem:fourier}
The Fourier transform of the Schwartz function $f \colon \R^n \times \R^n \to \R$ defined by
\[
f(x,y) = |x|^{2a+2k} |y|^{2b+2k} P_{2k}^{n}\mathopen{}\left(\frac{\langle x,y \rangle}{|x|\,|y|}\right)\mathclose{} e^{-\pi (|x|^2 + |y|^2)}
\]
is given by
\begin{equation*}\label{eq:fourier}
\widehat f(x,y) = \frac{a!b!}{\pi^{a+b}} L_a^{2k+\frac{n}{2}-1}(\pi |x|^2) L_b^{2k+\frac{n}{2}-1}(\pi |y|^2) |x|^{2k} |y|^{2k} P_{2k}^{n}\mathopen{}\left(\frac{\langle x,y \rangle}{|x|\,|y|}\right)\mathclose{} e^{-\pi (|x|^2 + |y|^2)}.
\end{equation*}
\end{lemma}

\begin{proof}
For $x,y \in S^{n-1}$, the addition theorem for Gegenbauer polynomials says that
\[
P_{2k}^n(x \cdot y) = \sum_m Y_{2k}^m(x) Y_{2k}^m(y),
\]
where the functions $Y_k^m$ are suitably normalized spherical harmonics in dimension $n$.
It follows that
\[
f(x,y) = \sum_m F_m(x) F_m(y),
\]
where
\[
F_m(x) = |x|^{2a+2k} Y_{2k}^l\mathopen{}\left(\frac{x}{|x|}\right)\mathclose{} e^{-\pi|x|^2} = |x|^{2a} Y_{2k}^m(x) e^{-\pi|x|^2}.
\]
The Fourier transform is therefore given by 
\[
\widehat f(x,y) = \sum_m \widehat F_m(x) \widehat F_m(y),
\]
where $\widehat F_m$ is the $n$-dimensional Fourier transform of $F_m$.

By Theorem~9.10.5 of \cite{andrews1999special}, the Fourier transform of $F_m(x)$ is
\[
\widehat F_m(x) = 2\pi i^{2k} |x|^{-2k-\frac{n}{2}+1} Y_{2k}^m(x) \int_0^\infty s^{2k+2a+\frac{n}{2}} e^{-\pi s^2} J_{2k+\frac{n}{2}-1}(2\pi|x|s) \, ds,
\]
where $J_\alpha$ is the Bessel function of the first kind of order $\alpha$.

By \cite[Section~12.1, Formula~(5)]{buchholz69},
\[
\int_0^\infty e^{-t^2} t^{2\xi+\mu+1} J_\mu(2t\sqrt{z}) \, dt = \frac{\xi!}{2} e^{-z} z^{\frac{\mu}{2}} L_\xi^\mu(z)
\]
for $\xi \in \N_0$ and $\mu \in \C$ with $\xi + \mathrm{Re}(\mu) > -1$. The variable substitutions $z = \pi |x|^2$ and $t = \sqrt\pi s$ give
\[
\int_0^\infty e^{-\pi s^2} s^{2\xi+\mu+1} J_\mu(2\pi |x|s) \, ds = \frac{\xi!}{2\pi^{\xi+1}} e^{-\pi |x|^2} |x|^\mu L_\xi^\mu(\pi |x|^2),
\]
and we then set
$\mu = 2k+\frac{n}{2}-1$ and $\xi = a$ to obtain
\begin{align*}
\int_0^\infty e^{-\pi s^2}& s^{2k+2a+\frac{n}{2}} J_{2k+\frac{n}{2}-1}(2\pi |x|s) \, ds \\
&= \frac{a!}{2\pi^{a+1}} e^{-\pi |x|^2} |x|^{2k+\frac{n}{2}-1} L_a^{2k+\frac{n}{2}-1}(\pi |x|^2).
\end{align*}
We conclude that
\begin{align*}
\widehat F_m(x) &= 2\pi i^{2k} |x|^{-2k-\frac{n}{2}+1} Y_{2k}^m(x) \frac{a!}{2\pi^{a+1}} e^{-\pi |x|^2} |x|^{2k+\frac{n}{2}-1} L_a^{2k+\frac{n}{2}-1}(\pi |x|^2)\\
&= i^{2k} \frac{a!}{\pi^a} L_a^{2k+\frac{n}{2}-1}(\pi |x|^2) Y_{2k}^m(x) e^{-\pi |x|^2}\\
&= (-1)^k \frac{a!}{\pi^a} L_a^{2k+\frac{n}{2}-1}(\pi |x|^2) |x|^{2k} Y_{2k}^m(x/|x|) e^{-\pi |x|^2},
\end{align*}
from which we can obtain the Fourier transform of $f$ by using the addition formula again.
\end{proof}

For the constraint $\widehat f \geq 0$ we require
\[
\mathcal F(p)(u,v,t) = s_8(u,v,t) + (u+v)s_{9}(u,v,t) + uv\,s_{10}(u,v,t) + t(uv-t)s_{11}(u,v,t),
\]
where $s_8,\dots,s_{11}$ are sum-of-squares polynomials of appropriate degree that satisfy $s_i(u,v,t) = s_i(v,u,t)$ for all $u,v,t$. Here $\mathcal F$ is the linear operator such that
\[
\mathcal F(p)(|x|^2,|y|^2,\langle x, y\rangle^2) e^{-\pi (|x|^2+|y|^2)}
\]
is the Fourier transform of 
\[
p(|x|^2,|y|^2,\langle x, y\rangle^2) e^{-\pi (|x|^2+|y|^2)},
\]
which is determined by Lemma~\ref{lem:fourier}.
Finally, for the constraint $\widehat f(0)=1$ we require
\[
\mathcal F(p)(0,0,0) = 1.
\]
Since $f$ is specified by positive semidefinite matrices and linear constraints, the problem of minimizing $f(0) = p(0,0,0)$ is a semidefinite program. As $d \to \infty$, we expect that this semidefinite program converges to the optimal bound $\Delta_{\mathrm{lat},3}$, and in any case it yields valid upper bounds for $\Delta_{\mathrm{lat},3}$.

\subsection{Semidefinite programming formulations}\label{sec:sdpformulation}

We can use standard techniques to write the sum-of-squares formulations from Sections~\ref{sec:param-trunc} and~\ref{sec:parlat} as semidefinite programs.  

For this we parametrize a univariate sum-of-squares polynomial of degree $2d$ as $v(x)^{\sf T} X v(x)$, where $X$ ranges over the positive semidefinite matrices of size $d+1$ and $v(x)$ is the column vector of monomials in $x$ up to degree $d$.

To parametrize a sum-of-squares polynomial $s(u,v,t)$ of degree $2d$ satisfying $s(u,v,t) = s(v,u,t)$ using positive semidefinite matrix variables, we note that
\[
\R[u,v,t] = \R[u+v, uv, t] \oplus (u-v) \R[u+v, uv, t],
\]
where $\R[u+v, uv, t]$ is the invariant ring.
Let $w_d(u,v,t)$ be a vector whose entries form a basis for the polynomials in $\R[u+v,uv,t]$ of total degree (in $u$, $v$, and $t$) at most $d$. It follows we can write a degree $2d$ sum-of-squares polynomial $q$ that satisfies $q(u,v,t) = q(v,u,t)$ for all $u,v,t$ as
\[
q(u,v,t) = w_d(u,v,t)^{\sf T} X_1 w_d(u,v,t) + (u-v)^2 w_{d-1}(u,v,t)^{\sf T} X_2 w_{d-1}(u,v,t),
\]
where $X_1$ and $X_2$ are positive semidefinite matrices, because the sum of the cross terms in $(u-v) \R[u+v, uv, t]$ must vanish.

Note that our semidefinite program for lattice bounds is completely different from that of Pendavingh and van Zwam \cite{pendavingh2007new}.  While their technique uses semidefinite relaxations for the problem of optimizing a polynomial over a finite-dimensional semi-algebraic set cut out by Korkin-Zolotarev inequalities, our technique uses semidefinite relaxation for an infinite-dimensional set of measures.

\section{Numerical results}\label{sec:numerical}

Our code and data are available from the MIT Libraries DSpace@MIT repository at \url{https://hdl.handle.net/1721.1/143590}.

\subsection{Results for the truncated three-point bounds}

To find good functions for the bound $\Delta^+_3$ we use the semidefinite programming formulation from Section~\ref{sec:param-trunc}. We fix $n$, $r$, $d_2$, $d_3$, and $d_\mathrm{taylor}$ and optimize for $R$ using Brent's method, which finds a local minimum. 
Table~\ref{table:summary} shows the results for $r=1$, $d_2=24$, $d_3=9$, and $d_\mathrm{taylor}=24$, while Table~\ref{table:3alldegrees} shows how much the bounds improve by increasing the degree $d_3$ from $6$ to $9$.

\begin{table}
\caption{The three-point bound $\Delta_3^+$ for different dimensions $n$ and polynomial degrees $d$. The row $\Delta_2$ gives the linear programming bound and the row LB (``lower bound'') gives the conjecturally optimal sphere packing density. All reported bounds are rigorous upper bounds on the sphere packing density.}
\label{table:3alldegrees}
\begin{center}
\begin{tabular}{ccccccc}
\toprule
$d$\textbackslash $n$ & $3$ & $4$ & $5$ & $6$ & $7$\\
\midrule
$\Delta_2$ & $0.779746762$ & $0.647704966$ & $0.524980022$ & $0.417673416$ & $0.327455611$\\
\midrule
$6$ & $0.773719142$ & $0.637500670$ & $0.513167025$ & $0.411786389$ & $0.322580299$\\
$7$ & $0.772302942$ & $0.636860131$ & $0.512732432$ & $0.410712688$ & $0.321690513$\\
$8$ & $0.770801756$ & $0.636332298$ & $0.512662102$ & $0.410354869$ & $0.321229615$\\
$9$ & $0.770270440$ & $0.636107333$ & $0.512645131$ & $0.410303282$ & $0.321147106$\\
\midrule
LB & $0.740480489$ & $0.616850275$ & $0.465257613$ & $0.372947545$ & $0.295297873$\\
\bottomrule
\end{tabular}
\end{center}
\end{table}

\subsection{Results for the lattice three-point bounds}

Table~\ref{table:summary} shows the  results for the lattice three-point bound computed using the parametrization described in Section~\ref{sec:parlat} with $d=13$. The table includes only dimensions $n \le 9$, because the values with $d=13$ become worse than the linear programming bound when $n > 9$.
Table~\ref{table:latalldegrees} and Figure~\ref{figure:latticedata} show how the bounds change for $d = 4,\dots,13$ and $3 \le n \le 7$.

\begin{table}
\caption{The lattice bound $\Delta_\mathrm{lat}$ for different dimensions $n$ and polynomial degrees $d$. The row $\Delta_2$ gives the linear programming bound and the row LB gives the optimal lattice sphere packing density. All reported bounds are rigorous upper bounds on the lattice sphere packing density.}
\label{table:latalldegrees}
\begin{center}
\begin{tabular}{ccccccc}
\toprule
$d$\textbackslash $n$ & $3$ & $4$ & $5$ & $6$ & $7$\\
\midrule
$\Delta_2$ & $0.779746762$ & $0.647704966$ & $0.524980022$ & $0.417673416$ & $0.327455611$\\
\midrule
$4$ & $0.788825992$ & $0.663869735$ & $0.544077206$ & $0.447400969$ & $0.375137855$\\
$5$ & $0.781889135$ & $0.653653322$ & $0.538153396$ & $0.435036532$ & $0.352071693$\\
$6$ & $0.767694937$ & $0.637355636$ & $0.528579742$ & $0.426524051$ & $0.341978846$\\
$7$ & $0.758431473$ & $0.624695218$ & $0.510775620$ & $0.420101441$ & $0.333233065$\\
$8$ & $0.752561887$ & $0.619935194$ & $0.499296524$ & $0.409441538$ & $0.329525254$\\
$9$ & $0.748788795$ & $0.617926232$ & $0.493598452$ & $0.397958350$ & $0.327756923$\\
$10$ & $0.746503469$ & $0.617173341$ & $0.490567118$ & $0.391751044$ & $0.320213657$\\
$11$ & $0.744364870$ & $0.616946076$ & $0.488407525$ & $0.388524317$ & $0.314748970$\\
$12$ & $0.742666838$ & $0.616879313$ & $0.486574603$ & $0.386245187$ & $0.311714989$\\
$13$ & $0.741505108$ & $0.616858175$ & $0.485079642$ & $0.383937277$ & $0.309974296$\\
\midrule
LB & $0.740480489$ & $0.616850275$ & $0.465257613$ & $0.372947545$ & $0.295297873$\\
\bottomrule
\end{tabular}
\end{center}
\end{table}

\begin{figure}
\begin{center}
\begin{tikzpicture}
\draw [decorate,decoration={brace,amplitude=10pt,raise=4pt},yshift=0pt]
(4.5,11.696201) -- (4.5,11.107207) node [black,midway,xshift=1cm] {$n=3$};
\draw [decorate,decoration={brace,amplitude=10pt,raise=4pt},yshift=0pt]
(4.5,9.7155745) -- (4.5,9.2527541) node [black,midway,xshift=1cm] {$n=4$};
\draw [decorate,decoration={brace,amplitude=10pt,raise=4pt},yshift=0pt]
(4.5,7.8747003) -- (4.5,6.9788642) node [black,midway,xshift=1cm] {$n=5$};
\draw [decorate,decoration={brace,amplitude=10pt,raise=4pt},yshift=0pt]
(4.5,6.2651012) -- (4.5,5.5942132) node [black,midway,xshift=1cm] {$n=6$};
\draw [decorate,decoration={brace,amplitude=10pt,raise=4pt},yshift=0pt]
(4.5,4.9118342) -- (4.5,4.4294681) node [black,midway,xshift=1cm] {$n=7$};
\draw (-0.1,11.832390)--(0.1,11.832390);
\draw (0,11.932390)--(0,11.732390);
\draw (0.4,11.728337)--(0.6,11.728337);
\draw (0.5,11.828337)--(0.5,11.628337);
\draw (0.9,11.515424)--(1.1,11.515424);
\draw (1,11.615424)--(1,11.415424);
\draw (1.4,11.376472)--(1.6,11.376472);
\draw (1.5,11.476472)--(1.5,11.276472);
\draw (1.9,11.288428)--(2.1,11.288428);
\draw (2,11.388428)--(2,11.188428);
\draw (2.4,11.231832)--(2.6,11.231832);
\draw (2.5,11.331832)--(2.5,11.131832);
\draw (2.9,11.197552)--(3.1,11.197552);
\draw (3,11.297552)--(3,11.097552);
\draw (3.4,11.165473)--(3.6,11.165473);
\draw (3.5,11.265473)--(3.5,11.065473);
\draw (3.9,11.140003)--(4.1,11.140003);
\draw (4,11.240003)--(4,11.040003);
\draw (4.4,11.122577)--(4.6,11.122577);
\draw (4.5,11.222577)--(4.5,11.022577);
\draw (-0.1,9.9580460)--(0.1,9.9580460);
\draw (0,10.058046)--(0,9.8580460);
\draw (0.4,9.8047998)--(0.6,9.8047998);
\draw (0.5,9.9047998)--(0.5,9.7047998);
\draw (0.9,9.5603345)--(1.1,9.5603345);
\draw (1,9.6603345)--(1,9.4603345);
\draw (1.4,9.3704283)--(1.6,9.3704283);
\draw (1.5,9.4704283)--(1.5,9.2704283);
\draw (1.9,9.2990279)--(2.1,9.2990279);
\draw (2,9.3990279)--(2,9.1990279);
\draw (2.4,9.2688935)--(2.6,9.2688935);
\draw (2.5,9.3688935)--(2.5,9.1688935);
\draw (2.9,9.2576001)--(3.1,9.2576001);
\draw (3,9.3576001)--(3,9.1576001);
\draw (3.4,9.2541911)--(3.6,9.2541911);
\draw (3.5,9.3541911)--(3.5,9.1541911);
\draw (3.9,9.2531897)--(4.1,9.2531897);
\draw (4,9.3531897)--(4,9.1531897);
\draw (4.4,9.2528726)--(4.6,9.2528726);
\draw (4.5,9.3528726)--(4.5,9.1528726);
\draw (-0.1,8.1611581)--(0.1,8.1611581);
\draw (0,8.2611581)--(0,8.0611581);
\draw (0.4,8.0723009)--(0.6,8.0723009);
\draw (0.5,8.1723009)--(0.5,7.9723009);
\draw (0.9,7.9286961)--(1.1,7.9286961);
\draw (1,8.0286961)--(1,7.8286961);
\draw (1.4,7.6616343)--(1.6,7.6616343);
\draw (1.5,7.7616343)--(1.5,7.5616343);
\draw (1.9,7.4894479)--(2.1,7.4894479);
\draw (2,7.5894479)--(2,7.3894479);
\draw (2.4,7.4039768)--(2.6,7.4039768);
\draw (2.5,7.5039768)--(2.5,7.3039768);
\draw (2.9,7.3585068)--(3.1,7.3585068);
\draw (3,7.4585068)--(3,7.2585068);
\draw (3.4,7.3261129)--(3.6,7.3261129);
\draw (3.5,7.4261129)--(3.5,7.2261129);
\draw (3.9,7.2986190)--(4.1,7.2986190);
\draw (4,7.3986190)--(4,7.1986190);
\draw (4.4,7.2761946)--(4.6,7.2761946);
\draw (4.5,7.3761946)--(4.5,7.1761946);
\draw (-0.1,6.7110145)--(0.1,6.7110145);
\draw (0,6.8110145)--(0,6.6110145);
\draw (0.4,6.5255480)--(0.6,6.5255480);
\draw (0.5,6.6255480)--(0.5,6.4255480);
\draw (0.9,6.3978608)--(1.1,6.3978608);
\draw (1,6.4978608)--(1,6.2978608);
\draw (1.4,6.3015216)--(1.6,6.3015216);
\draw (1.5,6.4015216)--(1.5,6.2015216);
\draw (1.9,6.1416231)--(2.1,6.1416231);
\draw (2,6.2416231)--(2,6.0416231);
\draw (2.4,5.9693753)--(2.6,5.9693753);
\draw (2.5,6.0693753)--(2.5,5.8693753);
\draw (2.9,5.8762657)--(3.1,5.8762657);
\draw (3,5.9762657)--(3,5.7762657);
\draw (3.4,5.8278648)--(3.6,5.8278648);
\draw (3.5,5.9278648)--(3.5,5.7278648);
\draw (3.9,5.7936778)--(4.1,5.7936778);
\draw (4,5.8936778)--(4,5.6936778);
\draw (4.4,5.7590592)--(4.6,5.7590592);
\draw (4.5,5.8590592)--(4.5,5.6590592);
\draw (-0.1,5.6270678)--(0.1,5.6270678);
\draw (0,5.7270678)--(0,5.5270678);
\draw (0.4,5.2810754)--(0.6,5.2810754);
\draw (0.5,5.3810754)--(0.5,5.1810754);
\draw (0.9,5.1296827)--(1.1,5.1296827);
\draw (1,5.2296827)--(1,5.0296827);
\draw (1.4,4.9984960)--(1.6,4.9984960);
\draw (1.5,5.0984960)--(1.5,4.8984960);
\draw (1.9,4.9428788)--(2.1,4.9428788);
\draw (2,5.0428788)--(2,4.8428788);
\draw (2.4,4.9163538)--(2.6,4.9163538);
\draw (2.5,5.0163538)--(2.5,4.8163538);
\draw (2.9,4.8032049)--(3.1,4.8032049);
\draw (3,4.9032049)--(3,4.7032049);
\draw (3.4,4.7212346)--(3.6,4.7212346);
\draw (3.5,4.8212346)--(3.5,4.6212346);
\draw (3.9,4.6757248)--(4.1,4.6757248);
\draw (4,4.7757248)--(4,4.5757248);
\draw (4.4,4.6496144)--(4.6,4.6496144);
\draw (4.5,4.7496144)--(4.5,4.5496144);
\draw[black!50!white](0,11.696201)--(4.5,11.696201);
\draw[black!50!white] (0,11.107207)--(4.5,11.107207);
\draw[black!50!white](0,9.7155745)--(4.5,9.7155745);
\draw[black!50!white] (0,9.2527541)--(4.5,9.2527541);
\draw[black!50!white](0,7.8747003)--(4.5,7.8747003);
\draw[black!50!white] (0,6.9788642)--(4.5,6.9788642);
\draw[black!50!white](0,6.2651012)--(4.5,6.2651012);
\draw[black!50!white] (0,5.5942132)--(4.5,5.5942132);
\draw[black!50!white](0,4.9118342)--(4.5,4.9118342);
\draw[black!50!white] (0,4.4294681)--(4.5,4.4294681);
\draw (-0.5,{15*0.275})--(-0.5,{15*0.825});
\draw (-0.55,{15*0.3})--(-0.45,{15*0.3});
\draw (-0.55,{15*0.4})--(-0.45,{15*0.4});
\draw (-0.55,{15*0.5})--(-0.45,{15*0.5});
\draw (-0.55,{15*0.6})--(-0.45,{15*0.6});
\draw (-0.55,{15*0.7})--(-0.45,{15*0.7});
\draw (-0.55,{15*0.8})--(-0.45,{15*0.8});
\draw (-0.55,{15*0.3}) node[left] {$0.3$};
\draw (-0.55,{15*0.4}) node[left] {$0.4$};
\draw (-0.55,{15*0.5}) node[left] {$0.5$};
\draw (-0.55,{15*0.6}) node[left] {$0.6$};
\draw (-0.55,{15*0.7}) node[left] {$0.7$};
\draw (-0.55,{15*0.8}) node[left] {$0.8$};
\draw (-1.4,{15*0.55}) node[rotate=90] {density};
\draw (0,{15*0.275}) node[below] {$d=4$};
\draw (4.5,{15*0.275}) node[below] {$d=13$};
\end{tikzpicture}
\end{center}
\caption{A graphical display of the data from Table~\ref{table:latalldegrees}. Each pair of horizontal lines shows the two-point bound and optimal lattice packing density for a given dimension $n$, while the crosses show the lattice bounds for degrees~$4$ through~$13$.}
\label{figure:latticedata}
\end{figure}
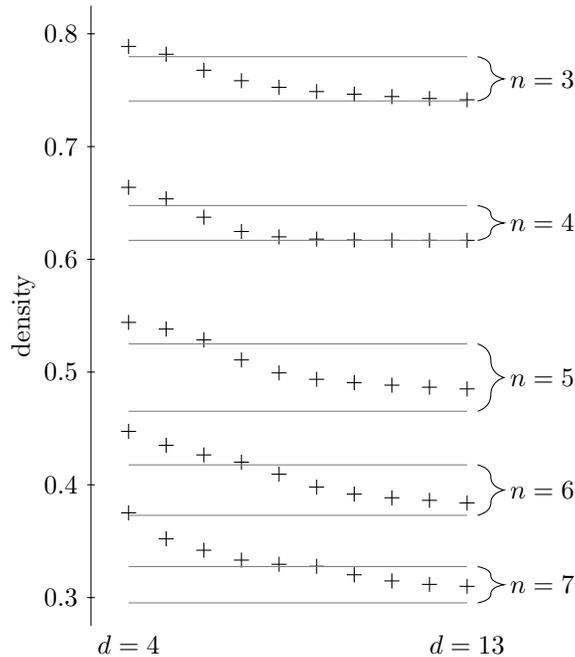

Note that the lattice three-point bound for dimension~$4$ agrees with the density of the $D_4$ lattice to  five decimal places when $d=13$, which is evidence that the bound is sharp in $\R^4$. While $D_4$ has long been known to be optimal among lattices, the existence of a special function producing a sharp bound would be remarkable. To formulate a detailed conjecture, let $D_4$ be the usual normalization as a root lattice, i.e., $\{x \in \Z^4 : x_1+x_2+x_3+x_4 \text{ is even}\}$, and let $D_4^*$ be its dual lattice.

\begin{conjecture}\label{conjectureR4}
There exists a Schwartz function $f \colon \R^4 \times \R^4 \to \R$ that satisfies the following conditions:
\begin{enumerate}
    \item $f(0,0) = 1$ and $\widehat{f}(0,0) = 4$,
    \item $f(x,y) \le 0$ whenever $(x,y) \ne (0,0)$ and each of $|x|$, $|y|$, $|x + y|$, and $|x-y|$ is either $0$ or at least $\sqrt{2}$,
    \item $\widehat{f}(x,y) \ge 0$ for all $x,y \in \R^4$,
    \item $f(x,y) = 0$ for $(x,y) \in D_4 \times D_4$, and \label{cond:4}
    \item $\widehat{f}(x,y) = 0$ for $(x,y) \in D_4^* \times D_4^*$. \label{cond:5}
\end{enumerate}
\end{conjecture}

In other words, $f$ satisfies the hypotheses of Theorem~\ref{thm:latticethreepoint} and the complementary slackness conditions needed to match the density of $D_4$. In fact, conditions~\eqref{cond:4} and~\eqref{cond:5} are redundant, because they follow from the other three conditions via Poisson summation over $D_4 \times D_4$.

It is unclear whether we should expect the lattice three-point bound to be sharp in any other cases, aside from those in which the two-point bound is already sharp. Table~\ref{table:latalldegrees} and Figure~\ref{figure:latticedata} show no sign of convergence to sharp bounds in dimensions~$5$ through~$7$. Dimension~$3$ could plausibly be sharp. We hesitate to conjecture it, because it is unclear why the convergence would be much slower than in dimension~$4$, but if we had to guess, we would guess that dimension~$3$ is also sharp.

As noted in Remark~\ref{rem:stronger}, the lattice three-point bound can be strengthened by including additional constraints. Table~\ref{table:altlatalldegrees} shows the effect of incorporating two more constraints.

\begin{table}
\caption{The three-point lattice bound with $S_{\mathrm{lat},n}$ replaced by $\{(x,y) : |x|,|y|,|x\pm y|, |x \pm 2y|, |2x \pm y| \in \{0\}\cup [r,\infty)\}$.}
\label{table:altlatalldegrees}
\begin{center}
\begin{tabular}{cccccc}
\toprule
$d$\textbackslash $n$ & $3$ & $4$ & $5$ & $6$ & $7$\\
\midrule
$\Delta_2$ & $0.779746762$ & $0.647704966$ & $0.524980022$ & $0.417673416$ & $0.327455611$\\
\midrule
$4$ & $0.789296829$ & $0.664404917$ & $0.544286855$ & $0.450821161$ & $0.375481755$\\
$5$ & $0.781909514$ & $0.653717975$ & $0.538210280$ & $0.435080014$ & $0.352637934$\\
$6$ & $0.767803262$ & $0.638061483$ & $0.528582888$ & $0.426546741$ & $0.342023570$\\
$7$ & $0.758421721$ & $0.624885125$ & $0.511042514$ & $0.420110428$ & $0.333262111$\\
$8$ & $0.752556256$ & $0.619949415$ & $0.499419261$ & $0.409580988$ & $0.329530413$\\
$9$ & $0.748489069$ & $0.617921450$ & $0.493599420$ & $0.397963691$ & $0.327792164$\\
$10$ & $0.744763159$ & $0.617170201$ & $0.490566217$ & $0.391765872$ & $0.320255703$\\
$11$ & $0.742104935$ & $0.616945200$ & $0.488398254$ & $0.388526826$ & $0.314781075$\\
$12$ & $0.741025543$ & $0.616879078$ & $0.486488035$ & $0.386228495$ & $0.311717562$\\
$13$ & $0.740655261$ & $0.616858082$ & $0.484733530$ & $0.383486370$ & $0.309988625$\\
\midrule
LB & $0.740480489$ & $0.616850275$ & $0.465257613$ & $0.372947545$ & $0.295297873$\\
\bottomrule
\end{tabular}
\end{center}
\end{table}

\subsection{Rigorous verification of the results}\label{sec:verification}

We can obtain explicit bounds by solving semidefinite programs numerically, but the results generally lack precise estimates for round-off error. To obtain rigorous bounds, we must prove that there exist exact solutions to these semidefinite programs with specified behavior. This task amounts to proving the existence of a solution to a semidefinite feasibility problem of the form
\begin{equation}\label{eq:sdpfeasibiliy}
X \succeq 0 \text{ and } \langle A_i, X \rangle = b_i \text{ for } i=1,\dots,m,
\end{equation}
i.e., the existence of a positive semidefinite matrix satisfying certain linear constraints.

For the cases that arise in our calculations, we can assume the constraints are linearly independent and a strictly feasible solution exists; that is, we assume there exists a positive definite matrix satisfying the linear constraints. We further assume we do not have direct access to the matrices $A_1,\dots,A_m$ and the vector $b$, but only to matrices $\tilde A_1,\dots, \tilde A_m$ and a vector $\tilde b$ whose entries are intervals containing the corresponding entries from $A_i$ and $b$. Equivalently, we can compute these matrices and vector rigorously using interval arithmetic.

Let $\bar A_1,\dots,\bar A_m$ and $\bar b$ be the matrices and vector whose entries are the midpoints of the intervals giving the corresponding entries in $\tilde A_1,\dots,\tilde A_m$ and $\tilde b$. It will be convenient to use the vectorized form, for which we use the symmetric vectorization map $\mathrm{svec}$, which gives a bijection from the symmetric $n \times n$-matrices to the column vectors containing $\binom{n+1}{2}$ entries.

We first solve the semidefinite feasibility problem 
\[
\mathrm{svec}^{-1}(x) \succeq 0, \quad \bar A x = \bar b 
\]
(here written in vectorized form) in floating point arithmetic using the high precision floating point semidefinite programming solver SDPA-GMP \cite{yamashita2012sdpa}. In practice this results in a vector $x^*$ with $\bar A x^* \approx \bar b$ and with $\mathrm{svec}^{-1}(x^*)$ positive definite. 

Then we perform a nonrigorous LU decomposition with pivoting in floating point arithmetic (using UmfpackLU \cite{davis2004lu}) to find a set $I$ of indices corresponding to a maximal set of linearly independent columns of $\bar A$. Consider the square linear system
\[
\tilde A_I y = \tilde b - \sum_{i \not \in I} x^*_i \tilde A_i,
\]
where $\tilde A_I$ is the matrix containing the columns of $\tilde A$ indexed by $I$, and $\tilde A_i$ is the $i$-th column. We solve this system rigorously in ball arithmetic using  the LU factorization in Arb \cite{johansson2017arb}, which results in a solution $\tilde y$ for which it is guaranteed that there exists a solution  $y$ to 
\[
A_I y = b - \sum_{i \not \in I} x^*_i A_i
\]
with $y_i \in \tilde y_i$ for all $i$. Then, for $\tilde x$ defined by
\[
\tilde x_i = \begin{cases} \tilde y_i & \text{if $i \in I$, and}\\
x_i^* & \text{otherwise,}\end{cases}
\] 
the existence of a solution $x$ to $Ax=b$ with $x_i \in \tilde x_i$ for all $i$ is guaranteed. 

Then we use Arb to compute the Cholesky factorization of $\mathrm{svec}^{-1}(\tilde x)$, which shows that any matrix $X$ satisfying $\mathrm{svec}(X)_i \in \tilde x_i$ for all $i$ is positive semidefinite. This shows $\tilde x$ contains a solution to \eqref{eq:sdpfeasibiliy}.

The above procedure can be used to obtain rigorous bounds from the semidefinite programming formulations of the three-point bounds. For this we first solve the semidefinite programming minimization problem in floating point arithmetic to get an approximation of the optimal objective value $v$. Then we solve the problem in floating point arithmetic as a semidefinite programming feasibility problem with the additional constraint where we set the objective function equal to $v + \varepsilon$ for some small constant $\varepsilon > 0$. By doing this the solver will find a point in the relative interior of the feasible set, which means that all eigenvalues that can be strictly positive (given the affine constraints) will be strictly positive. In the matrices used for the multivariable sum-of-squares characterizations some of the eigenvalues will be zero, but these correspond to diagonal entries of the matrices being zero. We then consider the semidefinite feasibility problem with the corresponding rows and columns removed, so that we can use (a block-diagonal version of) the above procedure to get rigorous bounds.

\appendix
\section{Existence of optimal functions}

\begin{proposition} \label{prop:LPattain}
For each dimension $n \ge 1$, there exists a continuous, integrable function that optimizes the linear programming bound for sphere packing in $\R^n$.
\end{proposition}

\begin{proof}
We prove this proposition by adapting the proof of Theorem~3 in \cite{gonccalves2017hermite}. Let $f_1, f_2, \dots$ be a sequence of auxiliary functions on $\R^n$ that converge to the optimal bound. In other words, each $f_j$ is continuous and integrable, $f_j(x) \le 0$ for $|x| \ge 1$, $\widehat{f}_j(y) \ge 0$ for all $y$, $\widehat{f}_j(0)=1$, and $\lim_{j \to \infty} f_j(0)$ is as small as possible. We would like to obtain a function $f$ satisfying the same conditions with $f(0) = \lim_{j \to \infty} f_j(0)$.

We start by observing a few inequalities. Because $f_j$ is positive definite, $\|f_j\|_\infty \le f_j(0)$. Furthermore, because $f_j(x) \le 0$ for $|x| \ge 1$,
\[
\|f_j\|_1 = \int_{B_1^n(0)} (|f_j(x)|+f_j(x)) \, dx - \int_{\R^n} f_j(x) \, dx \le 2 f_j(0) \vol(B_1^n(0)) - \widehat{f}_j(0).
\]
Thus, $\|f_j\|_\infty$ and $\|f_j\|_1$ are both bounded, as is $\|f_j\|_2$ because $\|f_j\|_2^2 \le \|f_j\|_\infty \|f_j\|_1$.

By the Banach-Alaoglu theorem, the unit ball in $L^2(\R^n)$ is weakly compact, and so by passing to a subsequence we can assume that $f_j$ converges in the weak topology on $L^2(\R^n)$ to a function $f \in L^2(\R^n)$. Furthermore, Mazur's lemma gives convex combinations $g_j$ of $f_j,f_{j+1},\dots,f_{k_j}$ that converge strongly to $f$. The conditions we have imposed on the functions $f_j$ are all convex, so without loss of generality we can assume that $\|f_j - f\|_2 \to 0$ as $j \to \infty$. By the Plancherel theorem, $\| \widehat{f}_j - \widehat{f}\|_2 \to 0$ as well. We can also assume that $f_j$ and $\widehat{f}_j$ converge to $f$ and $\widehat{f}$ pointwise almost everywhere, again by passing to a subsequence (see Theorem~3.12 in \cite{Rudin}).

By Fatou's lemma, $\widehat{f}$ must be integrable, since $\widehat{f}_j$ converges to $\widehat{f}$ pointwise almost everywhere, $\widehat{f}_j \ge 0$, and $\|\widehat{f}_j\|_1 = f_j(0)$ is bounded. Because $\widehat{f}$ is integrable, $f$ is continuous and bounded. It follows that $f$ is integrable on $B_1(0)$, and Fatou's lemma again shows that $f$ is integrable on $\R^n \setminus B_1^n(0)$, so $f$ is in $L^1(\R^n)$.

The inequalities $f(x) \le 0$ for $|x| \ge 1$ and $\widehat{f}(y) \ge 0$ for all $y$ hold by convergence almost everywhere and continuity. Furthermore, Fatou's lemma shows that
\[
f(0) = \int_{\R^n} \widehat{f}(y) \, dy \le \lim_{j \to \infty} \int_{\R^n} \widehat{f}_j(y) \, dy = \lim_{j \to \infty} f_j(0)
\]
and
\[
\widehat{f}(0) = \int_{\R^n} f(x) \, dx \ge \lim_{j \to \infty} \int_{\R^n} f_j(x) \, dx = 1
\]
(the integral over $B_1^n(0)$ converges by dominated convergence, and $f(x) \le 0$ elsewhere). We conclude that $f(0)/\widehat{f}(0) \le \lim_{j \to \infty} f_j(0)$, and thus $f$ is an optimal auxiliary function. In fact, we must have $f(0) = \lim_{j \to \infty} f_j(0)$ and $\widehat{f}(0)=1$, since otherwise $f$ would be better than optimal.
\end{proof}

\section*{Acknowledgements}

We are grateful to Nima Afkhami-Jeddi, Felipe Gon\c{c}alves, Tom Hartman, Nathan Kaplan, Rupert Li, Fernando Oliveira, Amirhossein Tajdini, and Frank Vallentin for helpful discussions.


\begin{thebibliography}{40}
\bibitem{ACHLT}
N.\ Afkhami-Jeddi, H.\ Cohn, T.\ Hartman, D.\ de~Laat, and
  A.\ Tajdini, \emph{High-dimensional sphere packing and the modular bootstrap}, J.
  High Energy Phys.\ (2020), no.~12, Paper No.\ 066, 44. \arXiv{2006.02560}
  \MR{4239386} \doi{10.1007/jhep12(2020)066}

\bibitem{andrews1999special}
G.~E.\ Andrews, R.\ Askey, and R.\ Roy, \emph{Special functions}, Encyclopedia of
  Mathematics and its Applications, vol.~71, Cambridge University Press,
  Cambridge, 1999. \MR{1688958} \doi{10.1017/CBO9781107325937}

\bibitem{BNOV}
C.\ Bachoc, G.\ Nebe, F.~M.\ de~Oliveira~Filho, and F.\ Vallentin, \emph{Lower
  bounds for measurable chromatic numbers}, Geom.\ Funct.\ Anal.\ \textbf{19}
  (2009), no.~3, 645--661. \arXiv{0801.1059} \MR{2563765}
  \doi{10.1007/s00039-009-0013-7}

\bibitem{bachoc2008new}
C.\ Bachoc and F.\ Vallentin, \emph{New upper bounds for kissing numbers from
  semidefinite programming}, J.\ Amer.\ Math.\ Soc.\ \textbf{21} (2008), no.~3,
  909--924. \arXiv{math/0608426} \MR{2393433}
  \doi{10.1090/S0894-0347-07-00589-9}

\bibitem{barvinok02}
A.\ Barvinok, \emph{A course in convexity}, Graduate Studies in Mathematics,
  vol.~54, American Mathematical Society, Providence, RI, 2002. \MR{1940576}
  \doi{10.1090/gsm/054}

\bibitem{buchholz69}
H.\ Buchholz, \emph{The confluent hypergeometric function with special emphasis
  on its applications}, Springer Tracts in Natural Philosophy, vol.\ 15,
  Springer-Verlag New York Inc., New York, 1969, Translated from the German by
  H.\ Lichtblau and K.\ Wetzel. \MR{0240343}

\bibitem{cohn2002new}
H.\ Cohn, \emph{New upper bounds on sphere packings {II}}, Geom.\ Topol.\ 
  \textbf{6} (2002), 329--353. \arXiv{math.MG/0110010} \MR{1914571}
  \doi{10.2140/gt.2002.6.329}

\bibitem{cohn2003new}
H.\ Cohn and N.\ Elkies, \emph{New upper bounds on sphere packings {I}}, Ann.\ of
  Math.\ (2) \textbf{157} (2003), no.~2, 689--714. \arXiv{math.MG/0110009}
  \MR{1973059} \doi{10.4007/annals.2003.157.689}

\bibitem{CKMRV2017}
H.\ Cohn, A.\ Kumar, S.~D.\ Miller, D.\ Radchenko, and M.\ Viazovska, \emph{The
  sphere packing problem in dimension 24}, Ann.\ of Math.\ (2) \textbf{185}
  (2017), no.~3, 1017--1033. \arXiv{1603.06518} \MR{3664817}
  \doi{10.4007/annals.2017.185.3.8}

\bibitem{rescaling2021}
H.\ Cohn and A.\ Salmon, \emph{Sphere packing bounds via rescaling}, preprint,
  2021. \arXiv{2108.10936}

\bibitem{cohn2019dual}
H.\ Cohn and N.\ Triantafillou, \emph{Dual linear programming bounds for sphere
  packing via modular forms}, Math.\ Comp.\ \textbf{91} (2021), no.~333,
  491--508. \arXiv{1909.04772} \MR{4350546} \doi{10.1090/mcom/3662}

\bibitem{cohn2014sphere}
H.\ Cohn and Y.\ Zhao, \emph{Sphere packing bounds via spherical codes}, Duke
  Math.\ J.\ \textbf{163} (2014), no.~10, 1965--2002. \arXiv{1212.5966}
  \MR{3229046} \doi{10.1215/00127094-2738857}

\bibitem{davis2004lu}
T.~A.\ Davis, \emph{Algorithm 832: {UMFPACK} {V}4.3---an unsymmetric-pattern
  multifrontal method}, ACM Trans.\ Math.\ Software \textbf{30} (2004), no.~2,
  196--199. \MR{2075981} \doi{10.1145/992200.992206}

\bibitem{Delsarte1972}
P.\ Delsarte, \emph{Bounds for unrestricted codes, by linear programming},
  Philips Res.\ Rep.\ \textbf{27} (1972), 272--289. \MR{314545}

\bibitem{DDFV}
C.\ Dobre, M.\ D\"{u}r, L.\ Frerick, and F.\ Vallentin, \emph{A copositive
  formulation for the stability number of infinite graphs}, Math.\ Program.\ 
  \textbf{160} (2016), no.~1-2, Ser.\ A, 65--83. \arXiv{1305.1819} \MR{3555382}
  \doi{10.1007/s10107-015-0974-2}

\bibitem{Francis}
M.\ Francis, \emph{A {D}ixmier-{M}alliavin theorem for {L}ie groupoids},
  preprint, 2020. \arXiv{2009.13760}

\bibitem{Garrett}
P.\ Garrett, \emph{{W}eil-{S}chwartz envelopes for rapidly decreasing
  functions},
  \url{https://www-users.cse.umn.edu/~garrett/m/fun/weil_schwartz_envelope.pdf},
  2004.

\bibitem{GMS2012}
D.~C.\ Gijswijt, H.~D.\ Mittelmann, and A.\ Schrijver, \emph{Semidefinite code
  bounds based on quadruple distances}, IEEE Trans.\ Inform.\ Theory \textbf{58}
  (2012), no.~5, 2697--2705. \MR{2952510} \doi{10.1109/TIT.2012.2184845}

\bibitem{gonccalves2017hermite}
F.\ Gon\c{c}alves, D.\ Oliveira~e Silva, and S.\ Steinerberger, \emph{Hermite
  polynomials, linear flows on the torus, and an uncertainty principle for
  roots}, J.\ Math.\ Anal.\ Appl.\ \textbf{451} (2017), no.~2, 678--711.
  \arXiv{1602.03366} \MR{3624763} \doi{10.1016/j.jmaa.2017.02.030}

\bibitem{HJ}
R.~A.\ Horn and C.~R.\ Johnson, \emph{Matrix analysis}, second ed., Cambridge
  University Press, Cambridge, 2013. \MR{2978290}

\bibitem{johansson2017arb}
F.\ Johansson, \emph{Arb: efficient arbitrary-precision midpoint-radius interval
  arithmetic}, IEEE Trans.\ Comput.\ \textbf{66} (2017), no.~8, 1281--1292.
  \MR{3681746} \doi{10.1109/TC.2017.2690633}

\bibitem{KaoYu}
W.-J.\ Kao and W.-H.\ Yu, \emph{Four-point semidefinite bound for equiangular
  lines}, preprint, 2022. \arXiv{2203.05828}

\bibitem{deLaat2020}
D.\ de~Laat, \emph{Moment methods in energy minimization: new
  bounds for {R}iesz minimal energy problems}, Trans.\ Amer.\ Math.\ Soc.\ 
  \textbf{373} (2020), no.~2, 1407--1453. \arXiv{1610.04905} \MR{4068268}
  \doi{10.1090/tran/7976}

\bibitem{de2014upper}
D.\ de~Laat, F.~M.\ de~Oliveira~Filho, and F.\ Vallentin,
  \emph{Upper bounds for packings of spheres of several radii}, Forum Math.
  Sigma \textbf{2} (2014), Paper No.\ e23, 42. \arXiv{1206.2608} \MR{3264261}
  \doi{10.1017/fms.2014.24}

\bibitem{de2018k}
D.\ de~Laat, F.~C.\ Machado, F.~M.\ de~Oliveira~Filho, and F.\ 
  Vallentin, \emph{$k$-point semidefinite programming bounds for equiangular
  lines}, Math.\ Program.\ \textbf{194} (2022), 533--567. \arXiv{1812.06045}
  \doi{10.1007/s10107-021-01638-x}

\bibitem{de2015semidefinite}
D.\ de~Laat and F.\ Vallentin, \emph{A semidefinite programming
  hierarchy for packing problems in discrete geometry}, Math.\ Program.\ 
  \textbf{151} (2015), no.~2, Ser.\ B, 529--553. \arXiv{1311.3789} \MR{3348162}
  \doi{10.1007/s10107-014-0843-4}

\bibitem{Li}
R.\ Li, \emph{Dual linear programming bounds for sphere packing via discrete
  reductions}, preprint, 2022. \arXiv{2206.09876}
  
\bibitem{LovaszShannon}
L.\ Lov\'{a}sz, \emph{On the {S}hannon capacity of a graph}, IEEE Trans.\ Inform.\ 
  Theory \textbf{25} (1979), no.~1, 1--7. \MR{514926}
  \doi{10.1109/TIT.1979.1055985}
  
\bibitem{MRR}
R.~J.\ McEliece, E.~R.\ Rodemich, and H.~C.\ Rumsey, Jr., \emph{The {L}ov\'{a}sz
  bound and some generalizations}, J.\ Combin.\ Inform.\ System Sci.\ \textbf{3}
  (1978), no.~3, 134--152. \MR{505587}

\bibitem{musin2006kissing}
O.~R.\ Musin, \emph{The kissing problem in three dimensions}, Discrete Comput.\ 
  Geom.\ \textbf{35} (2006), no.~3, 375--384. \arXiv{math/0410324} \MR{2202108}
  \doi{10.1007/s00454-005-1201-3}

\bibitem{musin2008kissing}
\bysame, \emph{The kissing number in four dimensions}, Ann.\ of Math.\ (2)
  \textbf{168} (2008), no.~1, 1--32. \arXiv{math/0309430} \MR{2415397}
  \doi{10.4007/annals.2008.168.1}

\bibitem{pendavingh2007new}
R.~A.\ Pendavingh and S.~H.~M.\ van Zwam, \emph{New {K}orkin-{Z}olotarev
  inequalities}, SIAM J.\ Optim.\ \textbf{18} (2007), no.~1, 364--378.
  \MR{2369065} \doi{10.1137/060658795}

\bibitem{pfender2007improved}
F.\ Pfender, \emph{Improved {D}elsarte bounds for spherical codes in small
  dimensions}, J.\ Combin.\ Theory Ser.\ A \textbf{114} (2007), no.~6, 1133--1147.
  \arXiv{math/0501493} \MR{2337242} \doi{10.1016/j.jcta.2006.12.001}

\bibitem{Rudin}
W.\ Rudin, \emph{Real and complex analysis}, third ed., McGraw-Hill Book Co.,
  New York, 1987. \MR{924157}

\bibitem{RudinFA}
\bysame, \emph{Functional analysis}, second ed., International Series in Pure
  and Applied Mathematics, McGraw-Hill, Inc., New York, 1991. \MR{1157815}

\bibitem{SardariZargar}
N.~T.\ Sardari and M.\ Zargar, \emph{New upper bounds for spherical codes and
  packings}, preprint, 2020. \arXiv{2001.00185}

\bibitem{SchrijverComparison}
A.\ Schrijver, \emph{A comparison of the {D}elsarte and {L}ov\'{a}sz bounds},
  IEEE Trans.\ Inform.\ Theory \textbf{25} (1979), no.~4, 425--429. \MR{536232}
  \doi{10.1109/TIT.1979.1056072}

\bibitem{Schrijver2005}
A.\ Schrijver, \emph{New code upper bounds from the {T}erwilliger algebra and
  semidefinite programming}, IEEE Trans.\ Inform.\ Theory \textbf{51} (2005),
  no.~8, 2859--2866. \MR{2236252} \doi{10.1109/TIT.2005.851748}

\bibitem{Schwartz}
L.\ Schwartz, \emph{Th\'{e}orie des distributions}, Publications de l'Institut
  de Math\'{e}matique de l'Universit\'{e} de Strasbourg, IX--X, Hermann, Paris,
  1966, Nouvelle \'{e}dition, enti\'{e}rement corrig\'{e}e, refondue et
  augment\'{e}e. \MR{0209834}

\bibitem{SimonVol4}
B.\ Simon, \emph{Operator theory}, A Comprehensive Course in Analysis, Part 4,
  American Mathematical Society, Providence, RI, 2015. \MR{3364494}
  \doi{10.1090/simon/004}

\bibitem{SteinWeiss}
E.~M.\ Stein and G.\ Weiss, \emph{Introduction to {F}ourier analysis on
  {E}uclidean spaces}, Princeton Mathematical Series, No.\ 32, Princeton
  University Press, Princeton, NJ, 1971. \MR{0304972}

\bibitem{Trinker}
H.\ Trinker, \emph{The triple distribution of codes and ordered codes}, Discrete
  Math.\ \textbf{311} (2011), no.~20, 2283--2294. \MR{2825675}
  \doi{10.1016/j.disc.2011.06.028}

\bibitem{viazovska2017}
M.~S.\ Viazovska, \emph{The sphere packing problem in dimension 8}, Ann.\ of
  Math.\ (2) \textbf{185} (2017), no.~3, 991--1015. \arXiv{1603.04246}
  \MR{3664816} \doi{10.4007/annals.2017.185.3.7}

\bibitem{yamashita2012sdpa}
M.\ Yamashita, K.\ Fujisawa, M.\ Fukuda, K.\ Kobayashi, K.\ Nakata, and M.\ Nakata,
  \emph{Latest developments in the {SDPA} family for solving large-scale
  {SDP}s}, Handbook on semidefinite, conic and polynomial optimization,
  Internat.\ Ser.\ Oper.\ Res.\ Management Sci., vol.\ 166, Springer, New York,
  2012, pp.~687--713. \MR{2894706} \doi{10.1007/978-1-4614-0769-0\_24}
\end{thebibliography}
\end{document}